\documentclass{amsart}
\usepackage{verbatim}
\usepackage{amssymb}
\usepackage{amsbsy}
\usepackage{amscd}
\usepackage{amsmath}
\usepackage{amsthm}
\usepackage{amsxtra}
\usepackage{latexsym}
\usepackage[mathscr]{eucal}
\usepackage{upgreek}
\usepackage{wasysym}
\usepackage[all,cmtip]{xy}
\usepackage[colorlinks]{hyperref}
\newcommand{\Ls}[1]{L_{\fing}(#1)}
\newcommand{\Ms}[1]{M_{\fing}(#1)}

\newcommand{\HC}{\mc{HC}}
\newcommand{\Vs}[1]{\Irr{#1\Lam_0}}

\newcommand{\D}{\mc{D}}

\newcommand{\Nil}{\mc{N}}
\newcommand{\BRS}[2]{H^{#1}_f(#2)}
\newcommand{\bw}[1]{\bigwedge\nolimits^{#1}}

\newcommand{\wh}{\widehat}

\newcommand{\mc}{\mathcal}
\newcommand{\mf}{\mathfrak}

\newcommand{\on}{\operatorname}
\newcommand{\KL}{\mathsf{KL}}

\newcommand{\Vg}[1]{V^{#1}(\fing)}%{V^{#1}(\fing)}

\newcommand{\finb}{\mathfrak{b}}
\newcommand{\finn}{\mathfrak{n}}

\newcommand{\isomap}{{\;\stackrel{_\sim}{\to}\;}}

\newcommand{\W}{\mathscr{W}}

\newcommand{\V}{V}%{\mathbb{V}}

\newcommand{\eW}{\widetilde{W}}

\newcommand{\affh}{\widehat{\mathfrak{h}}}

\newcommand{\affg}{\widehat{\mathfrak{g}}}

\newcommand{\fing}{\mathfrak{g}}

\newcommand{\finh}{\mathfrak{h}}
\newcommand{\finm}{\mathfrak{m}}

\newcommand{\Cl}{\mathscr{C}l}

\newcommand{\Wg}[1]{\W^{#1}(\fing, f)}

\newcommand{\Lamsemi}[1]{\bigwedge\nolimits^{\frac{\infty}{2}+#1}}

\newcommand{\Irr}[1]{L(#1)}
\newcommand{\IrrW}[1]{\mathbf{L}_{\W}(#1)}
\newcommand{\VermaW}[1]{\mathbf{M}_{\W}(#1)}

\newcommand{\U}{\mathcal{U}}

\newcommand{\BGG}{{\mathcal O}}
\newcommand{\N}{\mathbb{N}}
\newcommand{\Q}{\mathbb{Q}}
\renewcommand{\1}{{\mathbf{1}}}
\newcommand{\dual}[1]{{#1}^*}
\newcommand{\bra}{{\langle}}
\newcommand{\ket}{{\rangle}}

\newcommand{\nno}{\nonumber}
\newcommand{\Lam}{\Lambda}

\newcommand{\lam}{\lambda}
\newcommand{\ra}{\rightarrow}
\newcommand{\+}{\mathop{\oplus}}
\newcommand{\Z}{\mathbb{Z}}
\newcommand{\Mod}{\text{-}\mathrm{Mod}}

\newcommand{\cprime}{$'$}
\newcommand{\inv}{^{-1}}

\renewcommand{\*}{{\otimes}}
\newcommand{\C}{\mathbb{C}}
\newcommand{\che}{^{\vee}}

\theoremstyle{plain}
\newtheorem{Th}{Theorem}[section]
\newtheorem*{MainTh}{Main Theorem}
\newtheorem{Pro}[Th]{Proposition}

\newtheorem{Lem}[Th]{Lemma}

\newtheorem{Co}[Th]{Corollary}

\theoremstyle{definition}

\theoremstyle{remark}

\newtheorem{Rem}[Th]{Remark}
\newtheorem{Conj}{Conjecture}

\newcommand{\affW}{\widehat{W}}

\newcommand{\Var}{\textit{Var}}%{\mathcal{V}}%{\mathcal{V}}  %

\newcommand{\Zhu}{A}%{\mathscr{Z}hu}

\DeclareMathOperator{\Ann}{Ann}

\DeclareMathOperator{\im}{Im}

\DeclareMathOperator{\rank}{rk}
\DeclareMathOperator{\ch}{ch}
\DeclareMathOperator{\Res}{Res}

\DeclareMathOperator{\id}{id}

\DeclareMathOperator{\End}{End}
\DeclareMathOperator{\gr}{gr}
\DeclareMathOperator{\Ext}{Ext}

\DeclareMathOperator{\old}{old}
\DeclareMathOperator{\ad}{ad}
\DeclareMathOperator{\Ad}{Ad}
\DeclareMathOperator{\haru}{span}

\DeclareMathOperator{\Spec}{Spec}
\DeclareMathOperator{\Specm}{Specm}

\title{Rationality of  W-algebras: principal nilpotent cases
}
\author{Tomoyuki Arakawa}
\address{Research Institute for Mathematical Sciences, Kyoto University,
 Kyoto 606-8502 JAPAN}

\email{arakawa@kurims.kyoto-u.ac.jp}

\thanks{This work is partially  supported 
by JSPS KAKENHI Grant Number
No.\ 20340007 and No.\ 23654006.}

\begin{document}
\maketitle

\begin{abstract}
We prove the rationality of all the  minimal series
principal $W$-algebras
discovered by Frenkel, Kac and Wakimoto \cite{FKW92},
 thereby giving
 a new family of rational 
and $C_2$-cofinite vertex operator algebras.
A key ingredient in our proof   is
the study of   Zhu's algebra
of simple $W$-algebras
via
the quantized Drinfeld-Sokolov reduction.
We show that
the functor of taking Zhu's algebra 
commutes with the
reduction functor.
Using this general fact
we determine the
 maximal spectrums of the associated graded of Zhu's algebras
of vertex operator algebras associated with admissible representations
 of
affine Kac-Moody algebras
as well.
\end{abstract}

\section{Introduction}
Let $\W^k(\fing)=\W^k(\fing,f_{prin})$ be
the $W$-algebra
associated with
 a complex finite-dimensional simple Lie algebra $\fing$ and a principal 
nilpotent element $f_{prin}$ of $\fing$ 
at level $k$
\cite{FatLyk88,LukFat89,FF90}.
In \cite{Ara07}
we have confirmed 
the conjecture
of 
Frenkel, Kac and Wakimoto  \cite{FKW92}
on
the existence of modular invariant representations
of $\W^k(\fing)$ for an  appropriate level $k$.
These representations
are
called the  {\emph {minimal series representations}}
of $\W^k(\fing)$
since in the case that 
$\fing=\mf{sl}_2(\C)$
they
are precisely the 
minimal series representations \cite{BPZ84}
of the Virasoro algebra.
It has been  expected 
\cite{FKW92}
and  widely believed 
that
these representations of $\W^k(\fing)$
form
a minimal model of the
corresponding conformal field theory
in the sense of \cite{BPZ84}
as in the case that $\fing=\mf{sl}_2(\C)$.
In the language of  vertex operator algebras
this amounts to showing that  the
 vertex operator algebras
associated with
minimal series  representations of $W$-algebras
are {\em rational} and {\em $C_2$-cofinite}.
We have 
established 
the $C_2$-cofiniteness property
 previously  in \cite{Ara09b}.
The main purpose  of this paper is to resolve
the 
remaining rationality problem.

Denote
by
 $\W_k(\fing)$ the unique simple quotient 
of $\W^k(\fing)$ at a non-critical level $k$.
The vertex operator algebra $\W_k(\fing)$ is isomorphic to a  minimal series 
representation as a module over $\W^k(\fing)$ 
if and only if
\begin{align}
 k&+h_{\fing}\che=p/q\in \Q_{>0},\ p,q\in \N,\ (p,q)=1,\label{eq:form-of-k}\\
&\text{and }\begin{cases}
p\geq h_{\fing}\che,\ q\geq h_{\fing}  &\text{if }
(q,r\che)=1,\\
p\geq h_{\fing},\ q\geq r\che h_{{}^L\fing}\che
&\text{if }(q,r\che)=r\che,\nonumber
 \end{cases}
\end{align}
where
$h_{\fing}$ is
the Coxeter number of $\fing$,
$h_{\fing}\che$ is
  the dual Coxeter number of $\fing$,
${}^L\fing$ is the Langlands dual Lie algebra of $\fing$,
and $r\che$ is the maximal number of the edges in the Dynkin diagram
of $\fing$.
The central charge $c(k)$ of $\W_k(\fing)$ 
is given by the formula
\begin{align*}
c(p/q -h\che_{\fing})=l-12
\frac{|q\rho-p\rho\che|^2}{pq}
=- \frac{l((h_{\fing}+1)p-h_{\fing}\che q)
(r\che h_{{}^L\fing}\che p-(h_{\fing}+1)q)}{pq},
\end{align*}
where
$l$ is the rank of $\fing$,
$\rho$  is the half sum of positive roots of $\fing$
and
$\rho\che$
is the half sum of positive coroots of $\fing$.
 \begin{MainTh}
Let $k$ be as in (\ref{eq:form-of-k}).
The vertex operator algebra
$\W_k(\fing)$ is rational (and $C_2$-cofinite 
\cite{Ara09b}).
The set of isomorphism classes of
minimal series representations 
of $\W^k(\fing)$ 
%described in \cite{FKW92}
forms  the complete set of 
the isomorphism classes of simple  modules
over $\W_k(\fing)$.
 \end{MainTh}

Main Theorem has been proved  in \cite{BeiFeiMaz,Wan93}
in the case that
$\fing=\mf{sl}_2(\C)$  and
in \cite{DonLamTan04} in the case that $\fing=\mf{sl_3}(\C)$
and $k=5/4-3$ (or\footnote{There is the Feigin-Frenkel  duality
$\W_{p/q-h_{\fing}\che}(\fing)\cong \W_{q/r\che p-h_{{}^L\fing}\che}
({}^L\fing)$
for all $p,q\in \C^*$. (The details will be explained  elsewhere.)} $4/5-3$).

%Fusion rules among minimal series representations of $\W_k(\fing)$ has
%been calculated in \cite{FKW92}
%in the case that $\fing$ is simply laced using the Verlinde formula.%

\smallskip
Let us explain the outline of the proof 
of Main Theorem briefly.
A crucial step
in the proof
is the classification of
the simple
modules over
the simple quotient $\W_k(\fing)$.
For this purpose
it is sufficient \cite{Zhu96} to
determine 
Zhu's algebra 
of $\W_k(\fing)$.
We
carry out this
by studying 
Zhu's algebra of $W$-algebras
 via the quantized
Drinfeld-Sokolov reduction.
Since
this is a
 general argument we work in a 
 more  general setting:
Let $f$ be any nilpotent element of $\fing$,
$\W^k(\fing,f)$ the
(universal) $W$-algebra  associated with
$(\fing,f)$ at  level $k$.
By definition \cite{FF90,KacRoaWak03}
we have
\begin{align*}
 \W^k(\fing,f)=H^0_f(\Vg{k}),
\end{align*}
where
$\Vg{k}$ is the universal affine vertex algebra
associated with $\fing$ at level $k$ and
$H_f^{\bullet}(M)$ denotes the BRST cohomology
of
the generalized quantized Drinfeld-Sokolov reduction \cite{KacRoaWak03}
associated with
$(\fing,f)$ with coefficient in a $\Vg{k}$-module $M$.
We show
that
\begin{align}
 \Zhu(H^0_f(L))\cong H^0_f(\Zhu(L))
\label{eq:iso-of-Zhu-intro}
\end{align}
for any quotient  $L$ of $\Vg{k}$
at any level $k$ (in fact  we prove a stronger
assertion,
see Theorem \ref{Th:Zhu-algebra-correspondence}).
Here, for a conformal  vertex algebra $V$, $\Zhu(V)$ denotes 
Zhu's algebra\footnote{More precisely, $\Zhu(V)$ 
is the $L_0$-twisted Zhu's algebra in the sense
of \cite{De-Kac06} 
since $\Wg{k}$
 is $\frac{1}{2}\Z_{\geq 0}$-graded in general. 
It is  the usual Zhu's algebra for $f=f_{prin}$.} of $V$,
and 
$H^0_f(\Zhu(L))$ denotes the (finite-dimensional analogue of) BRST
cohomology
associated with $(\fing,f)$
with coefficient in $\Zhu(L)$,
which is identical to $\Zhu(L)_{\dagger}$
in Losev's notation \cite{Los11}, 
see  Section \ref{section:finite W-algebras}.

%By definition
%$\Zhu(\W_k(\fing,f))$ is a quotient of $\Zhu(\Wg{k})$.
%Hence
In the case that
$f=f_{prin}$ the  classification
problem is relatively simple since
  $\Zhu(\W^k(\fing,f_{prin}))\cong \mc{Z}(\fing)$ (\cite{Ara07}),
where $\mc{Z}(\fing)$ is the center of the universal enveloping algebra
$U(\fing)$ of $\fing$,
and hence,
$\Zhu(\W_k(\fing))$ is a quotient of
the commutative algebra $\mc{Z}(\fing)$.
Moreover under the assumption of Main Theorem we have shown in
\cite{Ara07}
that
\begin{align*}
\W_k(\fing)\cong H^0_{f_{prin}}(\Vs{k})
\end{align*}
as conjectured in \cite{FKW92},
where
 $\Vs{k}$  is  the unique simple quotient 
vertex algebra of $\Vg{k}$
which is an admissible representation \cite{KacWak89}
 as a $\affg$-module.
It follows 
from (\ref{eq:iso-of-Zhu-intro}) 
that
Zhu's algebra $A(\W_k(\fing))$ of 
$\W_k(\fing)$
is 
completely
determined by
$A(\Vs{k})$.
We 
  deduce
the classification result in Main Theorem 
from 
that of  admissible affine vertex algebras $L(k\Lam_0)$
recently
 obtained by the author 
in
\cite{A12-2}.

Once
the classification of simple modules
is established  it is straightforward to see that
there is no extension between two distinct 
simple $\W_k(\fing)$-modules 
from the general result on the representation theory
of $\W^k(\fing)$ achieved  in \cite{Ara07}.
Finally the fact that
simple $\W_k(\fing)$-modules do not 
admit non-trivial self-extensions
follows from the result of
 Gorelik and Kac \cite{GorKac0905}
who established  the complete reducibility of admissible representations
of $\affg$.

The isomorphism
(\ref{eq:iso-of-Zhu-intro})
has an application to affine vertex algebras 
as well:
It
enables  us to
determine  the variety
$\Var \Zhu(\Vs{k})$
associated with  Zhu's algebra
of any
 admissible affine vertex algebra $\Vs{k}$
(Theorem \ref{Th:vareity-of-admissible-affine}).
This result was
 announced in \cite{A12-2}.

\smallskip

The assertion of Main Theorem is a special case
of the  conjecture 
of Kac and Wakimoto \cite{KacWak08}
on the rationality of 
{\em exceptional  $W$-algebras}.
In subsequent papers we prove the rationality
of 
 a large family of $W$-algebras,
including all the exceptional $W$-algebras of type $A$,
generalizing the result of \cite{AraBP}.

\smallskip

This paper is organized as follows.
In Section \ref{sectopn:classicalBRST}
and Section \ref{section:finite W-algebras}
we reformulate some results of Ginzburg \cite{Gin08}
and Losev \cite{Los11}
in terms of BRST reduction for later purposes.
In Section \ref{section:Zhu-algberas}
we 
fix some notations
for vertex algebras
and
clarify the relationship
between Frenkel-Zhu's bimodules and
Zhu's  $C_2$-modules associated with vertex algebras.
In Section \ref{section:change-of-cinformal} we discuss 
the effect of shifts of conformal vector to Frenkel-Zhu's bimodules,
which is needed to describe  Frenkel-Zhu's bimodules associated with
$W$-algebras.
In Section \ref{section:affineVA} we collect some basic facts 
about affine vertex algebras
and study Zhu's $C_2$-modules and Frenkel-Zhu's bimodules
associated with 
objects in the
the Kazhdan-Lusztig parabolic full subcategory $\KL_k$ of $\BGG$
of $\affg$.
%Note that
%$\Vg{k}$ and its quotients belong to $\KL_k$.
In Section \ref{section:W-algebras}
we recall the definition of $W$-algebras
and some results from \cite{Ara09b}.
In Section \ref{section:Zhu'ss-bimodules-for-W}
 we show that
the functor of taking Frenkel-Zhu's bimodules
commutes with the
reduction functor on the category $\KL_k$. This result in particular
proves  (\ref{eq:iso-of-Zhu-intro}).
In Section \ref{section:varieties}
we recall the main result of \cite{A12-2}
and 
 determine
varieties 
$\Var \Zhu(\Vs{k})$
associated with  Zhu's algebras
of 
 admissible affine vertex algebras.
Finally we prove Main Theorem in Section \ref{section:proof}.

\smallskip

\subsection*{Acknowledgments}
The author wishes to thank
 Maria  Gorelik for
valuable discussions,
in particular, for giving him  a proof of Lemma 
\ref{Lem:extension}.
Some part of this work was done
while he was visiting
 Weizmann Institute, Israel, in May 2011,
Emmy Noether Center in Erlangen,
Germany
 in June 2011, 
Isaac Newton Institute for Mathematical Sciences,
UK, in 2011,
The University of Manchester,
University of Birmingham,
The University of Edinburgh,
 Lancaster University,
York University, UK,  in November 2011,
Academia Sinica, Taiwan, in December 2011,
Chern Institute of Mathematics,
Shanghai Jiao Tong University,
 China, in August 2012.
He is grateful to those institutes
for their hospitality.
Finally, he thanks
 the referees for the careful reading and useful comments.
\smallskip

\noindent {\em Notation}
Throughout this paper
the ground field is the complex number $\C$
and tensor products are meant to be as vector spaces
over $\C$ if not otherwise stated.

\section{The Slodowy slice and classical BRST reduction}
\label{sectopn:classicalBRST}
Let $R$ be a Poisson algebra.
Recall that a {\em Poisson module}
$M$ over 
$R$ is
a $R$-module  $M$ in the usual associative sense equipped with
a bilinear map
\begin{align*}
R\times M\ra M,\quad (r,m)\mapsto \ad r(m)=\{r,m\},
\end{align*}
which makes $M$ a Lie algebra module over $R$
satisfying 
\begin{align*}
 \{r_1,r_2 m\}=\{r_1,r_2\}m+r_2\{r_1,m\},\quad
\{r_1 r_2,m\}=r_1\{r_2,m\}+r_2\{r_1,m\}
\end{align*}
for $r_1,r_2\in R$, $m\in M$.
Let $R\on{-PMod}$ be the category of 
 Poisson modules
over $R$.

For any finite-dimensional Lie algebra
$\mf{a}$
the space
$\C[\mf{a}^*]=S(\mf{a})$ 
is a Poisson algebra
by the Kirillov-Kostant Poisson bracket.
A Poisson module over $\C[\mf{a}^*]$ is the same as 
a $\C[\mf{a}^*]$-module $M$ in the usual associative sense equipped with
a Lie algebra module 
structure $\mf{a}\ra \End M$,
$x\mapsto \ad(x)$, over $\mf{a}$
 such that
$\ad(x)(fm)=\{x,f\}m+ f\ad(x)(m)$
for $x\in \mf{a}$,
$f\in \C[\mf{a}^*]$,
$m\in M$.

Let $\fing$ be a finite-dimensional simple Lie algebra 
as in Introduction,
$(~|~)$ the normalized invariant inner product of $\fing$,
that is,
$1/2h\che_{\fing} \times $ the killing form of $\fing$.
Let
$\nu:\fing\isomap \fing^*$
be the isomorphism  
defined by
the   form  $(~|~)$.

Let
$f$ be a nilpotent element of $\fing$,
 $\{e,f,h\}$ 
an
$\mf{sl}_2$-triple
associated with $f$:
\begin{align*}
 [h,e]=2e,\quad 
[e,f]=h ,\quad [h,f]=-2f.
%\label{eq:triple}
\end{align*}
Set
\begin{align*}
\chi=\nu(f)\in \fing^*.
\end{align*}
The affine space
\begin{align*}
 \mc{S}_f=\nu(f+\fing^e)\subset \fing^*
\end{align*}
is called the \emph{Slodowy slice} at $\chi$ to $\Ad G.\chi$,
where
$\fing^e$ is the centralizer of $e$ in $\fing$
and  $G$ is the adjoint group of $\fing$.
It is known \cite{GanGin02}
that
the Kirillov-Kostant 
Poisson structure of $\fing^*$ restricts to $\mc{S}_f$.
Hence $\C[\mc{S}_f]$ is a Poisson algebra.

We have
\begin{align}
 \fing=\bigoplus_{j\in \frac{1}{2}\Z}\fing_j,
\quad
\fing_j=\{x\in \fing|\ad h(x)=2j x\}.
\label{eq:grading}
\end{align}
Put
\begin{align*}
\fing_{\geq 1}=\bigoplus_{j\geq 1}\fing_j\subset 
\fing_{>0}=\bigoplus_{j>0}\fing_j=\fing_{1/2}\+ \fing_{\geq 1}.
\end{align*}
Denote by $G_{>0}$ the unipotent subgroup of $G$
whose Lie algebra is $\fing_{>0}$.
By \cite[Lemma 2.1]{GanGin02}
the coadjoint action
gives the isomorphism
\begin{align}
 G_{>0}\times \mc{S}_f\isomap \chi+\fing_{\geq 1}^{\bot}
\label{eq:iso-gan-ginzburg}
\end{align}
of affine varieties,
where
$\fing_{\geq 1}^\bot$ is the annihilator of $\fing_{\geq 1}$
in $\fing^*$.

Consider the affine subspace
$\chi+\nu(\fing_{-1/2})$ of
$\fing_{>0}^*$.
We have
\begin{align*}
\C[\chi+\nu(\fing_{-1/2})]=\C[\fing_{>0}^*]/\bar I_{>0,\chi},
\end{align*}
where
$\bar I_{>0,\chi}$ is the Poisson  ideal of $\C[\fing_{>0}^*]$
generated by 
$
x-\chi(x)$
with $x\in \fing_{\geq 1}$.
%Since $\bar I_{>0,\chi}$
%is a Poisson submodule of
%$\C[\fing^*_{>0}]$,
%th$\chi+\nu(\fing_{-1/2})$ is a $M$-variety and 
%Since
%$N_{\bar \chi}$ is a Poisson ideal
%of $\C[\fing_{>0}^*]$,
%the Poisson structure
%of $\fing_{>0}^*$ restricts to $\chi+\nu(\fing_{-1/2})$.
%Under the identification
%$\C[\chi+\nu(\fing_{-1/2})]\cong  \C[\fing_{1/2}^*]$
The Poisson bracket   of
the quotient algebra 
 is given by
\begin{align*}
\{x,y\}=\chi([x,y])\quad \text{for }x,y\in \fing_{1/2}
\end{align*}
under the identification $\C[\chi+\nu(\fing_{-1/2})]\cong  \C[\fing_{1/2}^*]=S(\fing_{1/2})$.
As 
\begin{align}
\fing_{1/2}\times \fing_{1/2}\ra \C,
\quad (x,y)\mapsto \chi([x,y]),
\label{eq:symplectic}
\end{align}is a symplectic form,
it follows that
$\chi+\nu(\fing_{-1/2})$ is 
isomorphic to $T^* \C^{\dim \fing_{1/2}/2}$
as Poisson varieties. 

Let
\begin{align*}
 \mu:\fing^*\ra \fing_{\geq 1}^*
\end{align*}
be the restriction map.
Then $\mu$ is the moment map for the
action of the unipotent subgroup $G_{\geq 1}$
of $G$ whose Lie algebra if $\fing_{\geq 0}$.
We have
\begin{align}
\mu\inv(\chi+\nu(\fing_{-1/2}))=\chi+\fing_{\geq 1}^\bot.
\end{align}

Let $\{x_i| i=1,\dots, \dim \fing_{>0}
\}$ be a homogeneous basis of
$\fing_{>0}$ with respect to the grading (\ref{eq:grading}) such that
the first $\dim \fing_{1/2}$-elements $\{x_i|i=1,\dots, \dim
\fing_{1/2}\}$ forms a basis of $\fing_{1/2}$,
and let
$\{c_{ij}^k\}$ be the structure constant:
$[x_i,x_j]=\sum_{k}c_{ij}^k x_k$.
For $i=1,\dots, \dim \fing_{>0}$ let
$\bar \phi_i$ denote the image of 
$x_i$ under the natural  Poisson algebra homomorphism
$\C[\fing_{>0}^*]\twoheadrightarrow  \C[\chi+\fing_{1/2}^*]$.
By definition
\begin{align*}
\{\bar \phi_i,\bar \phi_j\}=\chi([x_i,x_j])\quad
\text{for } i=1,\dots,\dim \fing_{1/2},
\end{align*}
and $\bar \phi_i=\chi(x_i)$ for $i>\dim \fing_{1/2}$.

Let  $\Pi \fing_{>0}^*$  
denote the  space
$\fing_{>0}^*$ considered as a purely odd vector space,
$T^* \Pi\fing_{>0}^*$ the tangent bundle  of
$\Pi \fing_{>0}^*$
  which is a symplectic supermanifold.
Then
$\C[T^* \Pi \fing_{>0}^*]
$ is a 
Poisson superalgebra,
which is nothing but the exterior algebra
$\bw{\bullet}(\fing_{>0}^*\+\fing_{>0})
=\bw{\bullet}(\fing_{>0}^*)\*
\bw{\bullet}(\fing_{>0})$ (with an obvious Poisson superbracket).

For a Poisson module $M$ over $\C[\fing^*]$
set
\begin{align*}
 &\bar C(M)=M\* \C[\chi+\nu(\fing_{-\1/2})]\* 
\C[T^* \Pi \fing_{>0}^*]=\bigoplus_{p\in \Z} \bar C^{p}(M),\\
&\quad  \bar C^{p}(M)=
\bigoplus_{i-j=p}M\* \C[\chi+\nu(\fing_{-\1/2})]\*
\bw{i}(\fing_{>0}^*)\*
\bw{j}(\fing_{>0}).
\end{align*}
Then
$\bar C(\C[\fing^*])$ is naturally a graded
Poisson superalgebra,
and
$\bar C (M)$ is  a Poisson module over 
$\bar C(\C[\fing^*])$ (in an obvious ``super'' sense).
Set
\begin{align*}
 \bar d=\sum_{i=1}^{\dim \fing_{>0}}(x_i\* 1+ 1\* \bar \phi_i )
\* x_i^*
-1\* 1\* \frac{1}{2}\sum_{1\leq i,j,k\leq \dim\fing_{>0}}c_{ij}^k x_i^* x_j^* x_k\in 
\bar C^1(\C[\fing^*]),
\end{align*}
where $\{x_i^*\}\subset \fing_{>0}^*\subset \C[T^* \Pi \fing_{>0}^*]$ is the dual basis of $\{x_i\}$.

\begin{Lem}\label{Lem:square-of-bar-d-is-0}
$\bar d^2=0$.
\end{Lem}
Since
$\bar d$ is an odd element
it follows from Lemma \ref{Lem:square-of-bar-d-is-0}
that
$ (\ad \bar d)^2=0$ on any Poisson module over $\bar C(\C[\fing^*])$.
It  follows that
$(\bar C(\C[\fing^*]), \ad \bar d)$ 
is a differential graded superalgebra
and 
$(\bar C(M),\ad \bar d)$
is a module over 
the differential graded algebra
$(\bar C(\C[\fing^*]), \ad \bar d)$.
Let  $H^{\bullet}_f(M)$ be the
cohomology of the cochain complex
$(\bar C(M),\ad \bar d)$.
The space
$H^{\bullet}_f(\bar \C[\fing^*])$
 inherits the
$\Z$-graded Poisson superalgebra structure from
$\bar C(\C[\fing^*])$
and $\BRS{0}{M}$ is naturally a module over $\BRS{0}{\C[\fing^*]}$.
\begin{Th}[\cite{KosSte87,De-Kac06}, see also Theorem \ref{Th:vanishing-Poisson} below]
 We have $H^{i}_f(\C[\fing^*])=0$
for $i\ne 0$
and
$H^0_f(\C[\fing^*])\cong 
 \C[\mc{S}_f]$ as Poisson algebras.
\end{Th}

Let $\overline{\mc{HC}}$ be the
full subcategory of 
the category of 
$\C[\fing^*]\on{-PMod}$
%the Poisson modules 
%over $\C[\fing^*]$ 
consisting of modules on which
the Lie algebra action of $\fing$ is locally finite.
Denote by  $\bar I_{\chi}$   the ideal of $\C[\fing^*]$
generated by $y-\chi(y)$ with $y\in \fing_{\geq 1}$.
Then, for $M\in \overline{\HC}$,
$\bar I_{\chi}M$ is a Poisson submodule of
$M$ over $\C[\fing_{>0}^*]$.

The following assertion is a reformulation of a result 
of \cite{Gin08}.
\begin{Th}\label{Th:vanishing-Poisson}
For $M\in \overline{\mc{HC}}$,
we have
\begin{align*}
 H^i_f(M)\cong \begin{cases}
		(M/\bar I_{\chi} M)^{\ad \fing_{>0}}
&
\text{for }i=0,
\\0&\text{otherwise}
	       \end{cases}
\end{align*} 
In particular
the functor
\begin{align*}
\overline{\mc{HC}}\ra \C[\mc{S}_f]\on{-PMod},\quad
M\mapsto H_f^0(M),
\end{align*}
is 
exact,
and 
\begin{align*}
\on{supp}_{\C[\mc{S}_f]}\BRS{0}{M}=\mc{S}_f\cap \on{supp}_{\C[\fing^*]}(M)
\end{align*}
for a finitely generated object $M$ of $\overline{\HC}$.
\end{Th}
\begin{proof}
Since a cohomology functor commutes with injective limits we may assume
 that
$M$ is finitely generated.
Set
$\bar C=\bar C(M)$,
$\bar C^p=\bar C^p(M)$,
$\bar C^{ij}=M\*
\C[\chi+\fing_{\1/2}^*]\*
\bw{i}(\fing_{>0}^*)\*
\bw{-j}(\fing_{>0})\subset \bar C$,
so that
$\bar C^p=\bigoplus\limits_{i\geq 0,\ j\leq 0
\atop i+j=p}\bar C^{i,j}$.
 The differential
$\ad \bar d: \bar C^p\ra \bar C^{p+1}$ decomposes
as 
\begin{align*}
 \ad \bar d=\bar d_-\+\bar d_+,
\end{align*}
where
\begin{align}
 \bar d_-=&
 \sum_{i}(x_i\* \id+ \id \* \bar \phi_i ) \* \ad x_i^*,
\label{eq:bar-d-]}\\
\bar d_+=&
\sum_{i}(\ad x_i \* \id +\id \* \ad  \bar \phi_i )\* x_i^*
+\sum_{i,j,k}\id \* \id \* c_{ij}^k x_k  x_j^* \ad x_i^*\\
&-\id \* \id \* \frac{1}{2}\sum_{i,j,k}c_{ij}^k x_i^* x_j^* \ad x_k.
\nonumber
\label{eq:bar-d+}
\end{align}
Since
$\bar d_- \bar C^{i,j}\subset \bar C^{i,j+1}$,
$\bar d_+\bar C^{i,j}\subset \bar C^{i+1,j}$,
it follows that
\begin{align*}
 \{\bar d_-,\bar d_+\}=0,\quad \bar d_-^2=\bar d_+^2=0.
\end{align*}

Consider the spectral sequence
$E_r\Rightarrow H^{\bullet}_f(M)$
with
\begin{align*}
 E_1^{p,q}=H^{q}(\bar C^{p,\bullet},\bar d_-),
\quad E_2^{p,q}=H^p(E_1^{\bullet,q}
,\bar d_+).
\end{align*}
By (\ref{eq:bar-d-]}),
$H^{\bullet}(\bar C^{p,\bullet},\bar d_-)$
is the homology of the
Koszul complex 
of the $\C[\fing_{>0}^*]$-module 
$M\* \C[\chi+\nu(\fing_{-1/2})]\* \bw{p}(\fing_{>0}^*)
$ associated with the
sequence
$x_1,x_2,\dots, x_{\dim {\fing_{>0}}}$,
where
 $\C[\fing_{>0}^*]$ acts only on the first two factors.
Since
$\C[\chi+\nu(\fing_{-1/2})]$ is a free $ \C[\fing_{1/2}^*]$-module of rank $1$
it follows that
$H^{\bullet}(\bar C^{p,\bullet},\bar d_-)$
is isomorphic to the homology of the
Koszul complex 
of the  $\C[\fing_{\geq 1}^*]$-module 
$M\*\bw{p}(\fing_{>0}^*)$ associated with the
sequence
$x_{\dim \fing_{1/2}+1
}-\chi(x_{\dim \fing_{1/2}+1}),\dots, x_{\dim \fing_{>1}}
-\chi(x_{\dim \fing_{> 1}})$.
Hence thanks to  \cite[Corollary 1.3.8]{Gin08}
we have 
\begin{align}
 E_1^{\bullet,q}\cong \begin{cases}
		       (M/\bar I_{\chi}M)\* \bw{\bullet}(\fing_{>0}^*)&
\text{for }q=0,\\0&\text{for }q\ne 0.
		      \end{cases}
\end{align}
Hence from (\ref{eq:bar-d+})
we see that
$E_2^{\bullet,0}$ is isomorphic to the Lie algebra cohomology
$H^{\bullet}(\fing_{>0},M/\bar I_{\chi}M
 )$.

Now first consider the case 
 that $M=\C[\fing^*]$.
Since $
\C[\fing^*]/\bar I_{\chi}\cong \C[\chi+\fing_{\geq 1}^{\bot}]$
we have
$\C[\chi+\fing_{\geq 1}^{\bot}]=\C[G_{>0}]\*_{\C}\C[\mc{S}_f]$
by (\ref{eq:iso-gan-ginzburg}),
and thus,
\begin{align}
H^i(\fing_{>0},\C[\chi+\fing_{\geq 1}^{\bot}])\cong
\begin{cases}
 \C[\mc{S}_f]&\text{for }i=0,\\
0&\text{for }i>0.
\end{cases}
\label{eq:vanishing-slodowy-slices}
\end{align}
For a general module $M$
 the argument of  \cite[6.2]{GanGin02}
shows that
the multiplication map
\begin{align*}
\varphi:\C[\chi+\fing_{\geq 1}^{\bot}]
\*_{\C[\mc{S}_f]}(M/\bar I_{\chi}M
 )^{\ad \fing_{>0}}
\ra M/\bar I_{\chi}M
\end{align*}
is an isomorphism
of $\fing_{>0}$-module,
where
$\fing_{>0}$ acts only on the first factor $\C[\chi+\fing_{\geq 1}^{\bot}]$
of $\C[\chi+\fing_{\geq 1}^{\bot}]
\*_{\C[\mc{S}_f]}(M/\bar I_{\chi}M
 )^{\ad \fing_{>0}}$
and $\C[\mc{S}_f]$ acts on $(M/\bar I_{\chi}M)^{\ad \fing_{>0}}$ by the identification 
$\C[\mc{S}_f]
=(\C[\fing^*]/\bar I_{
\chi}\C[\fing^*])^{\ad \fing_{>0}}$.
Therefore
 (\ref{eq:vanishing-slodowy-slices}) gives that
\begin{align*}
 E_2^{p,q}\cong \begin{cases}
		(M/\bar I_{\chi}M)^{\ad \fing_{>0}}&\text{for }p=q=0,\\
0&\text{otherwise}.
		\end{cases}
\end{align*}
We conclude that
the spectral sequence collapses at $E_2=E_{\infty}$,
and the assertion follows.
\end{proof}

\section{Finite $W$-algebras
and equivalences of categories
via
BRST reduction}
\label{section:finite W-algebras}

Let $A$ be
  an associative algebra  over $\C$
equipped with an increasing $\frac{1}{2}\Z$-filtration $F_{\bullet}A$
such that
\begin{align}
F_p A\cdot  F_qA\subset F_{p+q}A,\quad 
[F_p A,F_qA]\subset F_{p+q-1}A.
\label{eq:compatible-filtration}
\end{align}Then the associated graded space
$\gr_F A=\bigoplus_{p\in \frac{1}{2}\Z}F_pA/F_{p-1/2}A$
is naturally a Poisson algebra.
We assume that
$\gr_F A$ is finitely generated as a ring.

Denote by   $A\on{-biMod}$ 
  the 
category
of $A$-bimodules.
Let $M$ be an object 
of  $A\on{-biMod}$ 
equipped with an increasing
filtration $F_{\bullet}M$ compatible with the one on
$A$,
that is,
\begin{align*}
 F_p A\cdot  F_q M\cdot  F_r A\subset F_{p+q+r}M,\quad [F_pA, F_q M]\subset F_{p+q-1}M.
\end{align*}
Then  $\gr_F M=\bigoplus_{p}F_pM/F_{p-1/2}M$
is naturally a Poisson module over $\gr_F A$.
The filtration $F_{\bullet}M$ is called {\em good}
if $\gr_F M$ is finitely generated over $\gr_F A$ in a usual associative
sense.
If this is the case we set
\begin{align*}
\Var M=\on{supp}(\gr_F M)\subset \Spec(\gr_F A),
\end{align*}
equipped with the reduced scheme structure.
It is well-known that
$\Var M$ is independent of the choice of a good filtration.

Let $F_{\bullet}U(\fing)$ be the standard PBW filtration of
$U(\fing)$:
\begin{align*}
F_{-1}U(\fing)=0,\quad
F_0 U(\fing)=\C,\quad 
F_p U(\fing)=\fing F_{p-1}U(\fing)+F_{p-1}U(\fing).
\end{align*}
Set
$F_p U(\fing)[j]=\{u\in U_{p}(\fing)|\ad h (u)=2 j u\}$,
where,
recall, $h$ is defined in 
Section \ref{sectopn:classicalBRST}.
Let
\begin{align*}
K_pU(\fing)=\sum\limits_{i-j\leq p}F_iU(\fing)[j].
\end{align*}
Then $K_{\bullet}U(\fing)$
is an
increasing,
exhaustive, 
separated filtration 
of $U(\fing)$
that satisfies \eqref{eq:compatible-filtration}.
The filtration 
$\{K_pU(\fing)\}$ is called the {\em Kazhdan filtration}.
The  associated graded
Poisson algebra
$\gr_K U(\fing)$
is naturally isomorphic to
$\C[\fing^*]$.

Let $M$ be a  $U(\fing)$-bimodule.
A Kazhdan filtration 
of $M$
is an increasing,
exhaustive, separated,
filtration
$K_{\bullet}M$ which is compatible with the Kazhdan filtration of $U(\fing)$.

Define 
\begin{align*}
I_{>0,\chi}=\sum_{x\in \fing_{\geq 1}}U(\fing_{>0})(x-\chi(x)).
\end{align*}
Then $I_{>0,\chi}$ is a two-sided ideal
of $U(\fing_{>0})$.
Set
\begin{align*}
 \D=U(\fing_{>0})/I_{>0,\chi},
\end{align*}
and 
let
\begin{align*}
\phi: U(\fing_{>0})\twoheadrightarrow  \D
\end{align*}
be the natural surjective algebra homomorphism,
$\phi_i=\phi(
x_i)$,
where $\{x_i\}$ is defined in section \ref{sectopn:classicalBRST}.
Then
\begin{align*}
[\phi_i,\phi_j]=\chi([x_i,x_j])\quad 
\text{for } i=1,\dots,\dim \fing_{1/2},
\end{align*}
and $\bar \phi_i=\chi(x_i)$ for $i>\dim \fing_{1/2}$.
It follows that
$\D$
%$U(\fing_{>0})/N_{\chi}$
is isomorphic
to 
the 
Weyl algebra
of rank $\dim \fing_{1/2}/2$.
Let $K_{\bullet}\D$ be the filtration of $\D$ induced by
$K_{\bullet}U(\fing)$,
that is,
$K_p\D
$
 the image of
$ K_pU(\fing)\cap U(\fing_{>0})$
 in 
$\D$.
 The  associated graded Poisson algebra
$\gr_K \D
$ is isomorphic to 
$\C[\chi+\nu(\fing_{-1/2})]$
which appeared in Section \ref{sectopn:classicalBRST}.

Denote by
$\Cl $ the
Clifford algebra
associated with
$\fing_{>0}\+\fing_{>0}^*$
 and the bilinear form
$\fing_{>0}\+\fing_{>0}^*\times \fing_{>0}\+\fing_{>0}^*\ra \C$,
$(x+f,x'+f')\mapsto f(x')+f'(x)$.
The algebra
$\Cl $ contains $\bw{\bullet}(\fing_{>0}^*)$
and $\bw{\bullet}(\fing_{>0})$ as its subalgebras
and the multiplication map
$\bw{\bullet}(\fing_{>0}^*)
\* \bw{\bullet}(\fing_{>0})\ra
\Cl $
is a linear isomorphism.
Let
$F_{\bullet}\Cl$ be the increasing
filtration 
of $\Cl$ defined by
$F_p\Cl=\bigoplus_{j\leq p} \bw{\bullet}(\fing^*_{>0})\*
\bw{j}(\fing_{>0})$.
Set
$F_p \Cl [j]=\{\omega\in F_p\Cl|\ad h(\omega)=2 j
\omega\}$,
and define  the filtration $K_{\bullet}\Cl$ by
\begin{align*}
K_p\Cl=\sum\limits_{i-j\leq p}F_i\Cl [j].
\end{align*}
We have $\gr_K\Cl\cong \C[T^* \Pi \fing_{>0}^*]$ as Poisson superalgebras.

Let
$\mathcal{HC}$ be the full subcategory of
 $U(\fing)\on{-biMod}$
consisting of modules on which
the adjoint $\fing$-action is locally finite.

For  $M\in \HC$,
let
\begin{align*}
& C(M)=M\* \D
\* \Cl =\bigoplus_{p\in \Z}C^p
 (M),\\
&\quad C^p(M)=\bigoplus_{i-j=p}M\* \D
\* \bw{i}(\fing_{>0}^*)\* \bw{j}(\fing_{>0}).
\end{align*}
Here we have used the linear isomorphism
$\Cl\cong \bw{\bullet}(\fing_{>0}^*)\* \bw{\bullet}(\fing_{>0})$.
The space
$C(M)$ is naturally a $\Z$-graded 
bimodule
over the $\Z$-graded superalgebra
 $C(U(\fing))$.

Set
\begin{align*}
 d=\sum_{i}(x_i\* 1+ 1\* \phi_i )
\* x_i^*
-1\* 1\* \frac{1}{2}\sum_{i,j,k}c_{ij}^k x_i^* x_j^* x_k\in 
C^1(U(\fing)).
\end{align*}

\begin{Lem}\label{Lem:square-of-d-is-0}
 $d^2=0$ in 
$C(U(\fing))$.
\end{Lem}
Since
$d$ is an odd element
it follows from Lemma \ref{Lem:square-of-d-is-0}
that
$ (\ad d)^2=0$ on $C(M)$.
By abuse of notation
we denote
by $H_f^{\bullet}(M)$ the cohomology of the cochain complex
$(C(M), \ad d)$.
Since
$(C(U(\fing)), \ad d)$ is a 
differential graded algebra
$H^{\bullet}_f(U(\fing))$ is naturally a
$\Z$-graded superalgebra and 
$H_f^{\bullet}(M)$ is naturally
a bimodule over $H_f^{\bullet}(U(\fing))$.

The {\em finite $W$-algebra}
\cite{Pre02}
associated with $(\fing,f)$
may be defined as 
the associative algebra
\begin{align*}
 U(\fing,f):=H_f^0(U(\fing))
\end{align*}
(\cite{DAnDe-De-07}, see \eqref{eq:usal-finite-W} below).

Let 
$K_{\bullet}M$ be 
a Kazhdan filtration 
of 
$M\in \HC$.
Set
\begin{align*}
 K_p C(M)=\sum_{p_1+p_2+p_3\leq p}K_{p_1}M\* K_{p_2}D\* K_{p_3}\Cl.
\end{align*}
When this is applied to $M=U(\fing)$,
$K_{\bullet}C(U(\fing))$ 
 defines an increasing,
exhaustive, separated filtration of $C(U(\fing))$
satisfying \eqref{eq:compatible-filtration}.
Note that $d\in K_1 C(U(\fing))$,
and thus,
$\ad d \cdot K_p C(U(\fing))\subset K_p C(U(\fing))$
and $\ad d$ defines a derivation of 
$\gr_K C(U(\fing))$.
By definition the differential graded algebra
$(\gr_K C(U(\fing)), \ad d)$
 is isomorphic to
$(\bar C(\C[\fing^*])
, \ad \bar d )$ , and
$\gr_K C(U(\fing^*))$ is isomorphic to $\bar C(\gr_K M)$
as Poisson modules over $\bar C( \C[\fing^*])$,
where 
$\bar C(\gr_K M)$ is the complex 
considered in Section
\ref{sectopn:classicalBRST}.

Let
 $K_{\bullet} H^{\bullet}_f(M
)$ be
the 
filtration of
$H_f^{\bullet}(M)$ induced from the filtration $K_{\bullet}C(M)$.
We have
\begin{align*}
 \gr_K \BRS{0}{U(\fing)}\cong H^0_f(\gr_K U(\fing))\cong \C[\mc{S}_f]
\end{align*}
as Poisson algebra (\cite{GanGin02,De-Kac06}). 
In fact,
we have the following more general assertion.
\begin{Th}\label{Th:vanishing-finite-dimensional}
\begin{enumerate}
 \item Let $M$ be an finitely generated object of $\HC$,
$K_{\bullet}M$ a good
Kazhdan-filtration of $M$.
Then
\begin{align*}
\gr_K H^i_f(M)\cong
		    H^i_f(\gr_K M)\cong  \begin{cases}
(\gr_K M/\bar I_{\chi}\gr_K M)^{\ad \fing_{>0}}
&\text{for }i=0,\\
0&\text{otherwise}
		   \end{cases}
\end{align*}as Poisson modules over 
$\C[\mc{S}_f]$.
In particular
\begin{align*}
 \Var\BRS{0}{M}=\Var M\cap \mc{S}_f.
\end{align*}
\item
We have 
$H^i_f(M)=0$ for $i\ne 0$,
$M\in \mc{HC}$.
In particular
the functor
\begin{align}
\mc{HC}\ra U(\fing,f)\on{-biMod},
\quad M\mapsto H^0_f(M),
\label{eq:losev's functor}
\end{align}is exact.
\end{enumerate}\end{Th}
\begin{proof}
(i)
By the assumption $\gr_K M$ is an object of
$\overline{\mc{HC}}$.
Moreover,
thanks to (the proof of) \cite[Lemma 4.3.3]{Gin08},
the filtration $K_{\bullet} C(M)$ is convergent in the sense of 
\cite{CarEil56}.
Hence the assertion follows immediately from Theorem \ref{Th:vanishing-Poisson}.
(ii)
Suppose  that  $M$ is finitely generated.
Then
$M$ admits a good Kazhdan filtration,
and hence,
$H^i_f(M)=0$ for $i\ne 0$.
But
this prove the
vanishing of all $M\in \mc{HC}$
since the cohomology functor commutes with injective limits.
\end{proof}

We shall now give yet another description of the 
functor \eqref{eq:losev's functor},
%$\HC\ra U(\fing,f)\on{-bimod}$,
%$M\mapsto \BRS{0}{M}$,
and show that \eqref{eq:losev's functor} is equivalent to the
functor constructed by Ginzburg \cite{Gin08}
and Losev \cite{Los11}, independently.

Choose a Lagrangian
subspace $l$ of $\fing_{1/2}$ with respect to the symplectic form
(\ref{eq:symplectic}),
and
let
\begin{align*}
 \finm=l\+\fing_{\geq 1}.
\end{align*}
Then $\finm$ is a nilpotent subalgebra of $\fing_{>0}$
and
the restriction of $\chi$ to $\finm$
is a character, that is,
$\chi([x,y])=0$ for $x,y\in \finm$.
Let $\{x'_i| i=1,\dots, \dim\finm\}$ be a basis of $\finm$,
$\{{x_i'}^*| i=1,\dots,\dim \finm\}$ the dual basis of $\finm^*$,
${c_{ij}^k}'$ 
the structure constants of $\finm$.

Let $\Cl_{\finm}$ 
Clifford algebra
associated with
$\finm\+\finm^*$
 and the natural bilinear form
on it.
For $M\in \mc{HC}$ set
\begin{align*}
& C(M)'=M\* \Cl_{\finm},\\
& d'=\sum_{i=1}^{\dim \finm}(x_i'+\chi(x_i'))
\* {x_i'}^*
-1\* 1\* \frac{1}{2}\sum_{1\leq i,j,k\leq \dim \finm}{c_{ij}^k}'
 {x_i'}^* {x_j'}^* x_k'\in 
C(U(\fing))'.
\end{align*}
Then
we have
$(d')^2=0$
and $(C(M'),\ad d')$ is  a cochain complex as well.
Denote by
 $H_f^{\bullet}(M)'$  the corresponding cohomology.

  \begin{Pro}\label{Pro:identification}
$ $
\begin{enumerate}
\item 
We have
an algebra isomorphism
$H_f^{0}(U(\fing))'
\cong U(\fing,f) $.
\item For $M\in \mc{HC}$ we have
$\BRS{i}{M}'=0$ for $i\ne 0$ and
$H_f^{0}(M)'\cong H^{0}_f(M) $ as modules over $U(\fing,f)$.
\end{enumerate}
 \end{Pro}
 \begin{proof}
We may assume that
  $M$ be a finitely generated as in the proof of Theorem \ref{Th:vanishing-finite-dimensional}.
Let $K_{\bullet}M$ be a good Kazhdan filtration.
In the same manner as Theorem \ref{Th:vanishing-finite-dimensional}
one can show that 
\begin{align*}
 \gr_K \BRS{i}{M}'\cong 
\begin{cases}
 (\gr_K M/\finm_{\chi}\gr_KM)^{\ad \finm}&\text{for }i=0,\\
0&\text{for }i\ne 0,
\end{cases}
\end{align*}
where $\finm_{\chi}$ is the ideal generated by $x-\chi(x)$ with
$x\in \finm$.
Since
the natural map
$(\gr_K M/\bar I_{\chi}\gr_K M)^{\ad \fing_{>0}}\ra
(\gr_K M/\finm_{\chi}\gr_K M)^{\ad \finm}$ 
is an isomorphism by the argument of
\cite[5.5]{GanGin02}
we have
\begin{align}
\gr_K \BRS{0}{M}\isomap \gr_K \BRS{0}{M}'
\label{eq:gr-iso}
\end{align}
as modules over $\C[\mc{S}_f]$.

Now in the same manner  as  in
\cite[3.2.5]{AKM}
one can construct a map
$\BRS{0}{M}\ra \BRS{0}{M}'$,
which induces the map \eqref{eq:gr-iso},
and hence must be an isomorphism.
For $M=U(\fing)$ 
this gives an algebra isomorphism 
$\BRS{0}{U(\fing)}\isomap \BRS{0}{U(\fing)}'$ 
and for a general $M $ this gives the assertion (ii).
 \end{proof}

Let $\C_{\chi}$ be the one-dimensional representation of $\finm$
defined  by the character $\chi$.
For $M\in \HC$ 
the space
\begin{align*}
\on{Wh}_{\finm}(M):=
M\*_{U(\finm)}\C_{\chi}
\end{align*}
is equipped with a $(U(\fing), U(\fing,f))$-bimodule structure.
Indeed, there is an obvious 
left $U(\fing)$-module structure on $\on{Wh}_{\finm}(M)$.
To see the right $U(\fing,f)$-module structure
consider the space
$M\* \bw{\bullet}(\finm)$,
which is naturally  a {\em right} module 
over $C(U(\fing))'=U(\fing)\*
\Cl_{\finm}$.
Under this right module structure the element $d'\in C(U(\fing))'$
gives $M\* \bw{\bullet}(\finm)$ the chain complex structure,
and this complex is identical to the 
Chevalley complex for calculating 
the Lie algebra $\finm$-homology 
$H_{\bullet}(\finm,M\* \C_{\chi})$
with coefficient in the diagonal $\finm$-module
$M\* \C_{\chi}$,
where $\finm$ acts on $M$ by $x m=-mx$.
The right $C(U(\fing))'$-action
on $M\* \bw{\bullet}(\finm)$
gives the right $U(\fing,f)$-action  on 
$H_{\bullet}(\finm,M\* \C_{\chi})$,
in particular on 
$H_{0}(\finm,M\* \C_{\chi})=\on{Wh}_{\finm}(M)$.
This action obviously commutes with the left $U(\fing)$-action.

By \cite{Gin08}, we have 
$H_i(\finm, M\* \C_{\-\chi})=0$ for $i\ne 0$,
$M\in \HC$,
and hence, the functor
\begin{align*}
\on{Wh}_{\finm}:\HC\ra (U(\fing),U(\fing,f))\on{-biMod},
\quad M\mapsto \on{Wh}_{\finm}(M)
\end{align*}
is  exact.

Let 
$\mc{C}$ be the full subcategory of
$\fing\on{-Mod}$
%the category of left $\fing$-modules 
consisting of objects 
on which
 $x-\chi(x)$ acts  locally nilpotently  for all $x\in \finm$.
Here, for any algebra $A$,
$A\on{-Mod}$ denotes the category of left $A$-modules.
Note that 
$\on{Wh}_{\finm}(M)$ with $M\in \mc{HC}$ belongs to
$\mc{C}$ when it is considered as a left $\fing$-module.

For an object 
$M$ of $\mc{C}$
consider the space
$M\* \bw{\bullet}(\finm^*)$
as  a  (left)
$C(U(\fing))'$-module.
The cochain complex 
$(M\* \bw{\bullet}(\finm^*),d')$ is identical to the Chevalley complex
for calculating 
Lie algebra $\finm$-cohomology
$H^{\bullet}(\finm, M\* \C_{-\chi})$
with coefficient in the diagonal $\finm$-module
$M\* \C_{-\chi}$.
It follows that 
$H^{\bullet}(\finm,M\* \C_{-\chi})$
is a module over
$U(\fing,f)$,
and
we have a functor
\begin{align*}
 \on{Wh}^{\finm}:\mc{C}\ra U(\fing,f)\on{-Mod},
\quad M\mapsto H^0(\finm, M\* \C_{-\chi}).
\end{align*}
By \cite{Skr02},
one knows that
$H^i(\finm, M\* \C_{-\chi})=0$ for $i>0$,
$M\in \mc{C}$, 
and $\on{Wh}^{\finm}$ defines  an 
{\em equivalence of categories}.

The following assertion can be proved in the same way as 
\cite[Theorem 2.4.2]{Ara07} using 
Proposition \ref{Pro:identification}.
\begin{Pro}\label{Pro:comparison-with-ginzbrug}
 For $M\in \HC$ we have
$\BRS{0}{M}\cong \on{Wh}^{\finm}(\on{Wh}_{\finm}(M)))$
as $U(\fing,f)$-bimodules.
\end{Pro}

Let
\begin{align*}
 Y=\on{Wh}_{\finm}(U(\fing))=U(\fing)\*_{U(\finm)}\C_{\chi}.
\end{align*}
Then 
by Proposition \ref{Pro:comparison-with-ginzbrug}
we obtain the usual realization of $U(\fing,f)$:
\begin{align}
 U(\fing,f)\cong \on{Wh}^{\finm}(Y)\cong \End_{U(\fing)}(Y)^{op}
\label{eq:usal-finite-W}.
\end{align}
The assignment  $U(\fing,f)\Mod\ra \mc{C}$,
$E\mapsto Y\*_{U(\fing,f)}E$,
gives a functor which is quasi-inverse to $\on{Wh}^m$  (\cite{Skr02}).

\begin{Rem}
By Proposition \ref{Pro:comparison-with-ginzbrug}
and \cite[3.5]{Los11},
it follows that
the functor
$\HC\ra U(\fing,f)\on{-biMod}$, $M\mapsto \BRS{0}{M}$,
coincides with the functor $\bullet_{\dagger}$
constructed by Losev \cite{Los11}.
This observation enables us to improve the main result of 
\cite{Ara08-a}; The details will appear elsewhere.
\end{Rem}

Let $I$  be a two-sided ideal of $U(\fing)$.
Then $U(\fing)/I$ is a quotient algebra,
and thus,
$\BRS{0}{U(\fing)/I}$ inherits the
algebra structure from $C(U(\fing)/I)$.
On the other hand,
the exact sequence
$0\ra I\ra U(\fing)\ra U(\fing)/I\ra 0$
induces the exact sequence
\begin{align*}
0\ra H^0_f(I)
\ra U(\fing,f)\ra H^0_f(U(\fing)/I)\ra 0
\end{align*}
by Theorem \ref{Th:vanishing-finite-dimensional}.
Hence
we have the algebra isomorphism
\begin{align}
H^0_f(U(\fing)/I)
\cong U(\fing,f)/H^0_f(I).
\label{eq:H(U/I)}
\end{align}

Let $\mc{C}^I$ denote the full subcategory of $\mc{C}$
consisting of objects which are annihilated by $I$.
\begin{Th}\label{eq:equivalence-of-categoroes}
For a two-sided ideal $I$
of $U(\fing)$ 
we have an equivalence of categories
\begin{align*}
 \mc{C}^I\cong H^0_f(U(\fing)/I)\on{-Mod},
\quad M\mapsto \on{Wh}^{\finm}(M).
\end{align*}
\end{Th}
\begin{proof}
By
\eqref{eq:H(U/I)},
$H^0_f(U(\fing)/I)\on{-Mod}$ can be identified
by (\ref{eq:H(U/I)}) with the
full subcategory of $U(\fing,f)\on{-Mod}$  consisting
objects $M$ which are annihilated by 
$H_f^0(I)$.
Therefore, thanks to Skryabin's equivalence,
it is  enough to check that
$\on{Wh}^{\finm}(M)\in \BRS{0}{U(\fing)/I}\on{-Mod}
$ for $M\in \mc{C}^I$,
and that $Y\*_{U(\fing,f)}E\in \mc{C}^I$ for $E\in \BRS{0}{U(\fing)/I}\on{-Mod}$.
The former is easy to see.
The latter follows from the proof of
\cite[Theorem 4.5.2]{Gin08}.
\end{proof}

\section{Frenkel-Zhu's bimodules and  Zhu's $C_2$-modules}
\label{section:Zhu-algberas}
Recall that a  vertex algebra 
is a vector space  $V$
equipped with
an element
$\1\in V$
called the vacuum,
$T\in \End(V)$,
and 
 a linear map
 \begin{align*}
Y(?,z):V\ra (\End V)[[z,z\inv]], 
\quad a\mapsto Y(a,z)=a(z)=\sum_{n\in \Z}
a_{(n)}z^{-n-1},
 \end{align*} 
such that
\begin{enumerate}
\item $\1(z)=\id_V$,
 \item $a_{(n)}b=0$ for $n\gg 0$,
$a,b\in V$,
and $a_{(-1)}\1=a$,
\item $(Ta)(z)=[T, a(z)]=\frac{d}{dz}a(z)$ for $a\in V$,
\item
$(z-w)^n [a(z),b(w)]=0$ in $\End (V)$
for $n\gg0$,
$a, b\in  V$.
\end{enumerate}

For a vertex algebra
$V$
we have the {\em Borcherds identity}
\begin{align*}
 \sum_{i=0}^{\infty}&\begin{pmatrix}
		     p\\i
		    \end{pmatrix}
(a_{(r+i)}b)_{(p+q-i)}=
\sum_{i=1}^{\infty}(-1)^i\begin{pmatrix}
			  r\\i
			 \end{pmatrix}
(a_{(p+r-i)}b_{(q+i)}-(-1)^r
b_{(q+r-i)}a_{(p+i)})\nno
\end{align*}
in $\End  V$
for all $p,q,r\in \Z$,
$a,b,c\in V$.

A module  over  a vertex algebra
$V$  is a vector space  $M$
equipped with a linear map
 \begin{align*}
Y^M(?,z):V \ra (\End M)[[z,z\inv]], 
\quad a\mapsto a^M(z)=\sum_{n\in \Z}
a_{(n)}^M z^{-n-1},
 \end{align*} 
such that
$Y^M(\1,z)=\id_M$,
$a_{(n)}^Mm=0$ for $n\gg 0$,
$a\in V$, $m\in M$,
and 
\begin{align*}
 \sum_{i=0}^{\infty}&\begin{pmatrix}
		     p\\i
		    \end{pmatrix}
(a_{(r+i)}b)^M_{(p+q-i)}=
\sum_{i=1}^{\infty}(-1)^i\begin{pmatrix}
			  r\\i
			 \end{pmatrix}
(a_{(p+r-i)}^Mb_{(q+i)}^M-(-1)^r
b_{(q+r-i)}^Ma_{(p+i)}^M)\nno
\end{align*}
in $\End  M$
for all $p,q,r\in \Z$,
$a,b,c\in V$.
In particular $V$ itself if a module over $V$
called the {\em adjoint module}.
Let $V\on{-Mod}$ be the abelian  category of $V$-modules.
Below if no confusion arises we write $a_{(n)}$ for $a_{(n)}^M$.

For a $V$-module $M$ 
set
 \begin{align*}
C_2(M):=\haru_{\C}\{a_{(-2)}m| a\in V, m\in M\}.
 \end{align*}
{\em Zhu's $C_2$-algebra} \cite{Zhu96}
of $V$ is by definition the space
\begin{align*}
R_V=:V/C_2(V)
\end{align*}
equipped with the
Poisson algebra structure given by
\begin{align*}
 \bar a \cdot \bar b=\overline{a_{(-1)}b},\quad
\{\bar a,\bar b\}=\overline{a_{(0)}b}\quad\text{for }a,b \in V,
\end{align*}
where $\bar a=a+C_2(V)$.
{\em Zhu's $C_2$-module} of $M$
is the space $M/C_2(M)$ equipped with the  Poisson module structure
over  
$R_V$ given by
\begin{align*}
 \bar a\cdot \bar m=\overline{a_{(-1)}m},\quad 
\{\bar a,\bar m\}=\overline{a_{(0)}m}\quad\text{for }a\in V,\ m\in M.
\end{align*}

A vertex algebra  $V$ is called {\em finitely strongly generated} if $R_V$ is finitely
generated
as a ring;
it is called {\em rational} if
any $V$-module  is completely reducible;
it is called {\em $C_2$-cofinite}
if Zhu's $C_2$-algebra $R_V$ is finite-dimensional.
The $C_2$-cofiniteness
 condition is equivalent to the lisse condition in the sense of
\cite{BeiFeiMaz} (\cite{Ara12}).

A vertex algebra
$V$ is called {\em conformal}
if it is equipped with a vector $\omega\in V$,
called the {\em conformal  vector},
such that the corresponding
field $Y(\omega,z)=\sum_{n\in \Z}L_n z^{-n-2}$
satisfies the relation
\begin{align*}
 &[L_m,
 L_n]=(m-n)L_{m+n}+\frac{(m^3-m)\delta_{m+n,0}}{12}c_V\quad\text{for
 some }c_V\in\C,\\
&L_{-1}=T,\\
&L_0\text{ is diagonalizable on }V.
\end{align*}
The number $c_V$ is called the {\em central charge} of $V$.

In this paper we assume that a vertex algebra $V$ is 
%finitely strongly generated, 
conformal
and
$\frac{1}{2}\Z$-graded\footnote{This is because 
$W$-algebras are $\frac{1}{2}\Z_{\geq 0}$-graded
in general. However since 
principal $W$-algebras are $\Z_{\geq 0}$-graded 
 it is enough to
consider the $\Z$-graded case
in order to prove Main Theorem.} 
 with respect to $L_0$:
\begin{align*}
V=\bigoplus_{d\in \frac{1}{2}\Z}V_{d},\quad V_{d}=\{a\in V| L_0 a=d a\}.
\end{align*}For a homogeneous elements $a\in V$
we denote by
 $\on{wt}(a)$ 
the eigenvalue of $L_0$ on $a$.

A $V$-module $M$
is called {\em graded} if 
\begin{align*}
M=\bigoplus_{d\in \C}M_d,
\quad M_d=\{m\in M|  (L_0-d)^r m=0,\ r\gg 0\};
\end{align*}
it is called {\em positively graded}
if in addition
there exists a finite set $\{d_1,\dots,d_r\}\subset \C$
such that $M_d=0$ unless $d\in \bigcup_{i=1}^r(d_i+\frac{1}{2}\Z_{\geq
0})$.
If $V$ is $C_2$-cofinite any finitely generated $V$-module is positively
graded (\cite{AbeBuhDon04}).
Let $V\on{-gMod}$ be the abelian full subcategory of
$V\on{-Mod}$ consisting of  positively graded $V$-modules,
$\on{Irr}(V)$  the set of isomorphism classes of simple objects of 
$V\on{-gMod}$.

Let $\Zhu(V)$ be the ($L_0$-twisted) {\em Zhu's algebra}
 of $V$ 
(\cite{FreZhu92,De-Kac06}).
By definition
  \begin{align*}
  \Zhu(V)=V/O(V),
 \end{align*}
where $O(V)$ is
the subspace of $V$
spanned by the vectors
\begin{align*}
a\circ b:=\sum_{i\geq 0}\begin{pmatrix}
			  \on{wt}(a)\\ i
			 \end{pmatrix}a_{(i-2)}b
\end{align*}with homogeneous vectors
 $a, b\in V$.
The multiplication $*$ 
of $\Zhu(V)$  given  by
\begin{align*}
a* b=
\sum_{i\geq 0}\begin{pmatrix}
			  \on{wt}(a)\\ i
			 \end{pmatrix}a_{(i-1)}b.
\end{align*}

Let $M$ be a 
$V$-module.
{\em Frenkel-Zhu's bimodule} \cite{FreZhu92}
associated to
 $M$
is the
bimodule
$\Zhu(M)$
over $\Zhu(V)$
defined by
  \begin{align*}
  \Zhu(M)=M/O(M),
 \end{align*}
where $O(M)$ is
the subspace of $M$
spanned by the elements
\begin{align*}
a\circ m:=\sum_{i\geq 0}\begin{pmatrix}
			  \on{wt}(a)\\ i
			 \end{pmatrix}a_{(i-2)}m
\end{align*}with homogeneous vectors  $a\in V$
and
 $m\in M$.
The bimodule structure of
$\Zhu(M)$
is 
given by
\begin{align}
a* m=
\sum_{i\geq 0}\begin{pmatrix}
			  \on{wt}(a)\\ i
			 \end{pmatrix}a_{(i-1)}m,
\quad 
m* a=
\sum_{i\geq 0}\begin{pmatrix}
			  \on{wt}(a)-1\\ i
			 \end{pmatrix}a_{(i-1)}m.
\label{eq:bimodule-structure-of-A(M)}
\end{align}
Note that
\begin{align}
 a*m-m*a= \sum_{i\geq 0}\begin{pmatrix}
			      \on{wt}(a)-1\\ i
			     \end{pmatrix}a_{(i)}m.
\label{eq:adjoing-action-on-Zhu-algebra-on-bimodule}
\end{align}

\begin{Lem}[{\cite[Proposition 1.5.4]{FreZhu92}}]
\label{Lem:FreZhu}
The assignment
$M\mapsto \Zhu(M)$
defines a right exact functor from 
$V\on{-Mod}$ to 
$\Zhu(V)\on{-biMod}$.
 \end{Lem}

\smallskip

Zhu's $C_2$-algebra 
$R_V$
and Zhu's algebra
$\Zhu(V)$ 
are  related as follows:
Set 
\begin{align*}
V_{\leq p}=\bigoplus_{d\leq p}V_{d}
\end{align*}
and let  $F_p \Zhu (V)$ be the image of
$V_{\leq p}$ in $\Zhu(V)$.
Then
$F_{\bullet} A(V)$
defines an increasing,
exhaustive
$\frac{1}{2}\Z$-filtration
of $\Zhu(V)$
satisfying (\ref{eq:compatible-filtration}) (\cite{Zhu96}).
(In the cases that we will consider in this paper 
the filtration $F_{\bullet}A(V)$
will be  separated as well; this is true, for instance,
if $V$ is positively graded.)
On the other hand
the grading  of $V$
induces the grading of $R_V$:
$R_V=\bigoplus_{p\in \frac{1}{2}\Z}(R_V)_{p}$,
where $(R_V)_{p}$ is the image
of $V_{p}$ in $R_V$:
$(R_V)_{p}\cong V_{p}/C_2(V)_{p}$,
$C_2(V)_{p}=C_2(V)\cap V_{p}$.
The linear map
\begin{align*}
(R_V)_{p}\ra F_p\Zhu(V)/F_{p-1/2}\Zhu(V),\quad
a+ C_2(V)_{p}\mapsto a +O(V)\cap V_{\leq p}
+V_{\leq p-1/2}
\end{align*}defines a surjective homomorphism
\begin{align}\label{eq:surjective-to-Zhu}
\pi_V: R_V\twoheadrightarrow  \gr_F \Zhu(V)
\end{align}
of graded Poisson algebras
(\cite[Proposition 3.2]{ALY}).
It follows that
$A(V)$ is finite-dimensional if $V$ is $C_2$-cofinite.

For a 
graded $V$-module
$M=\bigoplus_{d\in  \C }M_{d}$,
there is a similar relation between
$M/C_2(M)$ and $\Zhu(M)$ as well:
Set
\begin{align*}
M_{\leq p}=\bigoplus_{d\in  p-\frac{1}{2}\Z_{\geq 0}}M_{d},
\end{align*}
and let  $F_p \Zhu(M)$ be the image of
$M_{\leq p}$ in $\Zhu(M)$.
Then
the space 
 $\gr_F \Zhu(M)=\bigoplus_{p\in \C}F_p \Zhu(M)/F_{p-1/2}\Zhu(M)$
is a graded Poisson module over 
$\gr_F \Zhu(V)$,
and hence over $R_V$ by (\ref{eq:surjective-to-Zhu}).

The following assertion can be proved in the same manner as 
\cite[Proposition 3.2]{ALY}.
\begin{Lem}\label{Lem:surjective-hom-ZHu's-bimodule}
Let $M$ be a  graded
$V$-module.
The linear map
$M_p/C_2(M)_p\ra F_p\Zhu(M)/F_{p-1/2}\Zhu(M)$,
$m+ C_2(M)_{p}\mapsto m +O(M)\cap M_{\leq p}
+M_{\leq p-1/2}$,
defines a
 surjective homomorphism 
\begin{align*}
\pi_M:M/C_2(M)\twoheadrightarrow \gr_F \Zhu(M)
\end{align*}
  of Poisson modules over 
$R_V$.
Here
$C_2(M)_p=C_2(M)\cap M_p$.
\end{Lem}

Now assume for a moment that 
$V$ is $\Z_{\geq 0}$-graded with respect to $L_0$.
Let
 $\U(V)=\bigoplus_{d\in \Z}\U(\V)_d$ be the
current algebra \cite{FreZhu92,MatNagTsu05} of $V$,
which is a degreewise complete graded topological algebra.
Then a $V$-module is the same as 
a continuous representation of
$\U(V)$.
Since
\begin{align}
\Zhu(V)\cong \U(V)_0/\overline{\sum_{p>0}\U(V)_{p}\U(V)_{-p}}
\label{eq:zhu-iso}
\end{align}
(\cite{NagTsu05}),
where $\overline{U}$ denotes the degreewise closure of $U$,
an $\Zhu(V)$-module $E$ can be regarded as a module over 
$\U(V)_{\leq 0}
:=\bigoplus\limits_{p\leq 0}\U(V)_p$ on which
$\U(V)_{p}$, $p<0$, acts trivially.
Set
\begin{align}
M_V(E):=\U(V)\*_{\U(\V)_{\leq 0}}E\in V\on{-gMod},
\label{eq:def-of-Verma}
\end{align}
and
let $L_V(E)$ be the
unique simple  quotient of 
$M_V(E)$.
By Zhu's theorem \cite{Zhu96}
we have
\begin{align}
 \on{Irr}(V)=
\{L_V(E)|  E\in \on{Irr}(\Zhu(V))\},
\label{Zhus-theorem}
\end{align}
where,
for any algebra
$A$,
$\on{Irr}(A)$
denotes the set of isomorphism classes of simple objects of 
$A\on{-Mod}$.

\section{The effect of shifts of conformal vector to Frenkel-Zhu's bimodules}
\label{section:change-of-cinformal}
Let $V$ be a $\frac{1}{2}\Z$-graded conformal vertex algebra
with conformal vector $\omega$.
Suppose that
there exists an element $\xi\in V$ that satisfies  the conditions
\begin{align*}
 L_n\xi=\delta_{n,0}\xi,\quad \xi_{(n)}\xi=\kappa
 \delta_{n,1}\1\quad\text{for }n\in \Z_{\geq 0}
\end{align*}
with some $\kappa\in \C$, and that
$\xi_{(0)}$ acts semisimply on $V$  with eigenvalues in $\Z$.
Then one can ``shift'' the conformal vector
$\omega$ by $\frac{1}{2}L_{-1}\xi$ to obtain a new conformal vector.
Namely
\begin{align*}
 \omega_{\xi}:=\omega+\frac{1}{2}\xi_{(-2)}\1
\end{align*}
also defines a conformal vector of $V$, with central charge
$c_{new}=c_{old}-3\kappa$,
where $c_{\old}$ is the central charge of $V$ with respect to $\omega$.

Although the definition of Zhu's algebra and Frenkel-Zhu's bimodules depend on
the choice of a conformal vector,
the above shift of a conformal vector does not
change the 
structure of Zhu's algebra
nor Frenkel-Zhu's bimodules as we show below:
for a $V$-module $M$
let
$\Zhu^{new}(M)$ (temporary)
 denote
Frenkel-Zhu's bimodule of $M$
with respect  to the conformal vector $\omega_{\xi}$
and let $\Zhu^{\old}(M)$
(temporary)
denote
Frenkel-Zhu's bimodule 
with respect  to the conformal vector $\omega$.

Let $\Delta(z)$ be  Li's $\Delta$-operator \cite{Li97}
associated with  $\xi$:
\begin{align*}
 \Delta(z)=z^{\frac{\xi_{(0)}}{2}}\exp(\sum_{n\geq 1}
\frac{\xi_{(n)}}{-2n}(-z)^n).
\end{align*}
 \begin{Pro}\label{Pro:change-of-grading-does-not-change-the-grading}
$ $

\begin{enumerate}
 \item  The map
$V\ra V$,
$a\mapsto \Delta(1)a$, induces an algebra isomorphism
\begin{align*}
\Zhu^{old}(V)\isomap \Zhu^{new}(V).
\end{align*}
 \item  
Let $M$ be a 
$V$-module  on which $\xi_{(0)}$ acts semisimply.
Then  the  map
$M\ra M$,
$m\mapsto \Delta(1)m$, induces an
$A^{old}(V)(\cong A^{new}(V))$-bimodule isomorphism
\begin{align*}
\Zhu^{old}(M)\isomap \Zhu^{new}(M).
\end{align*}
\end{enumerate}
 \end{Pro}
Proposition \ref{Pro:change-of-grading-does-not-change-the-grading}
follows from the following lemma.
\begin{Lem}
Let $M$ be a 
$V$-module  on which $\xi_{(0)}$ acts semisimply.
Then 
\begin{align*}
\Delta(1)(a\circ_{old} m)=(\Delta(1)a)\circ_{new}(\Delta(1)m),\\
\Delta(1)(a*_{old} m)=(\Delta(1)a)*_{new}(\Delta(1)m),\\
\Delta(1)(m*_{old} a)=(\Delta(1)m)*_{new}(\Delta(1)a)
\end{align*}for $a\in V$,
$m\in M$.
Here $\circ_{old}$ and  $*_{old}$ 
(respectively, $\circ_{new}$ and $*_{new}$)
are
operations (\ref{eq:bimodule-structure-of-A(M)})
with respect to the grading defined by
$L_{0,old}$ (respectively, $L_{0,new}$).
Here $Y(\omega,z)=\sum_{n\in \Z}L_{n,old}z^{-n-2}$,
$Y(\omega_{\xi},z)=\sum_{n\in \Z}L_{n,new}z^{-n-2}$.
\end{Lem}
\begin{proof}
 Let $m$ be a homogeneous vector of $M$
such that $\xi_{(0)}m=2 \lam m$.
Then
$\on{wt}(m)_{new}=\on{wt}(m)_{old}-\lam$,
where $\on{wt}(m)_{new}$
and  $\on{wt}(m)_{old}$ denote 
the eigenvalue of $L_{0,new}$ and $L_{0,old}$ on $m$,
respectively.
Write
\begin{align*}
\exp(\sum\limits_{n\geq 1}
\frac{\xi_{(n)}}{-2n}(-z)^{-n})
=\sum\limits_{n\geq 0} u_n z^{-n}
\end{align*} with $u_n\in \C[\xi_{(1)},\xi_{(2)},\dots,]$.
Since
we have
\begin{align}
 \Delta(1)Y(a,z)= Y(\Delta(z+1)a,z)\Delta(1)
\end{align}
for any $a\in V$
by \cite[Proposition 3.2]{Li97},
we have
\begin{align*}
 \Delta(1)(a\circ_{old} m)
=\Delta(1)\Res_{z=0} (Y(a,z)\frac{(z+1)^{\on{wt}(a)_{old}}}{z^2}m)
\\= \Res_z (\Delta(1)Y(a,z)\frac{(z+1)^{\on{wt}(a)_{old}}}{z^2}m)\\
=\Res_{z=0} (Y(\Delta(z+1)a,z)\frac{(z+1)^{\on{wt}(a)_{old}}}{z^2}\Delta(1)m)
\\=\sum_{n\geq 0}\Res_{z=0}(Y(u_n
 a,z)\frac{(z+1)^{\on{wt}(a)_{old}+\lam-n}}{z^2}\Delta(1)m)
\\=\sum_{n\geq 0}\Res_{z=0}(Y(u_n
 a,z)\frac{(z+1)^{\on{wt}(u_n a)_{new}}}{z^2}\Delta(1)m)
=(\Delta(1)a)\circ_{new} (\Delta(1)m).
\end{align*}
The proof of the other equalities is similar.
\end{proof}

\section{Affine vertex algebras}
\label{section:affineVA}
Let 
$\affg$ be 
the non-twisted affine Kac-Moody algebra
associated with $\fing$
and $(~|~)$:
\begin{align*}
\affg=\fing[t,t\inv]\+ \C K.
\end{align*}
The  commutation  relations  of $\affg$ are
 given by
\begin{align*}
 &[xt^m,yt^n]=[x,y]t^{m+n}+m\delta_{m+n,0}(x|y)K\quad \text{for }x,y\in \fing,
m,n\in\Z,
\\&[K,\affg]=0.
\end{align*}
We consider $\fing$
as a subalgebra of $\affg$
by the embedding $\fing\hookrightarrow \affg$,
$x\mapsto x t^0$.

For $k\in \C$
define 
\begin{align*}
 \Vg{k}=U(\affg)\*_{U(\fing[t]\+ \C K)}\C_k,
\end{align*}
where 
$\C_k$ is the one-dimensional representation of  $\fing[t]\+ \C K$ on which 
$\fing[t]$ acts trivially and $K$ acts as a multiplication by $k$.
There is a unique vertex algebra 
structure on $\Vg{k}$
such that
$\1:=1\* 1$ is the vacuum
and
\begin{align*}
Y(x t\inv \1,z)=x(z):=\sum_{n\in \Z}(x t^n)z^{-n-1}
\end{align*}
for $x\in \fing$.
The vertex algebra $V(\fing)$ is called the {\em universal affine vertex
algebra}
associated with $\fing$ at level $k.$

A $\Vg{k}$-module is the same as
a smooth $\affg$-module
of level $k$,
where by a smooth $\affg$-module $M$
we mean a $\affg$-module
$M$ such that
 $(xt^n)m=$  for  $n\gg 0$, $x\in \fing$,
$m\in M$.

We have  
\begin{align}
C_2(M)=
\fing[t\inv]t^{-2}M
\label{eq:c2-of-affine-module}
\end{align}
 for a  $\Vg{k}$-module $M$.
It follows that
the assignment $x\mapsto \overline{(xt\inv)\1}$,
$x\in \fing$, 
gives the isomorphism of Poisson algebras
\begin{align}
\C[\fing^*]\isomap  R_{\Vg{k}}=\Vg{k}/\fing[t\inv]t^{-2}\Vg{k}.
\label{eq:vareity-of-vgk}
\end{align}
%where
%$\fing^*$ is
%equipped with the Kirillov-Kostant Poisson structure.
We will identify $R_{\Vg{k}}$ with $\C[\fing^*]$ through the above isomorphism.
The 
Poisson 
module structure of 
$M/C_2(M)=M/\fing[t\inv]t^{-2}M$
over $\C[\fing^*]$
is  then given by
\begin{align*}
 x \cdot \bar m= \overline{(xt\inv) m},
\quad \{x, \bar m\}=\overline{(xt^0)m}
\end{align*}
for $x\in \fing$,
$m\in M$.

We will assume that $k$ is non-critical,
that is, $k\ne -h_{\fing}\che$,
unless otherwise stated,
although this condition is not essential.
The standard conformal vector
$\omega_{\fing}$ of $\Vg{k}$ is given by the  Sugawara   construction:
\begin{align*}
\omega_{\fing}=\frac{1}{2(k+h_{\fing}\che)}\sum_{i}(X_i t\inv)(X^i t\inv)\1,
\end{align*}
where
 $\{X_i\}$ is a basis of $\fing$,
$\{X^i\}$ the  dual bases with respect to $(~|~)$.
This gives a  $\Z_{\geq 0}$-grading on $\Vg{k}$.

We have \cite{FreZhu92}
the natural isomorphism
of algebras
\begin{align}
 \Zhu(\Vg{k})\cong U(\fing).
\end{align}
This  can be also seen using \eqref{eq:zhu-iso}
from the fact that
the current algebra of $\Vg{k}$ is isomorphic to
the standard degreewise completion \cite{MatNagTsu05}
$\widetilde{U_k(\affg)}$
of $U_k(\affg):=U(\affg)/(K-k\id)$.
For a $\fing$-module $E$,
we have
\begin{align}
 M_{\Vg{k}}(E)\cong U(\affg)\*_{U(\fing[t]\+ \C K )}E,
\label{eq:induced module}
\end{align}
where $E$ is considered as a $\fing[t]\+ \C K$-modules
on which $K$ acts as the multiplication by $k$
and $\fing[t]t$ acts trivially.

Let $N_k(\fing)$ be the unique maximal  ideal
of $\Vg{k}$.
Then 
\begin{align*}
 \Vs{k}:=\Vg{k}/N_k(\fing)
\end{align*}
is a simple vertex algebra called the {\em (simple) affine vertex
algebra}
associated with $\fing$
at level $k$.

Let
$\KL_k$ be the full subcategory of the category
of $\Vg{k}\on{-gMod}$
consisting of objects $M$
on which
 $\fing\subset \affg$ acts locally finitely.
By  (\ref{eq:induced module}),
$M_{\Vg{k}}(E)$ is an object of $\KL_k$
for
 a finite-dimensional $\fing$-module $E$.

The following assertion is clear.
\begin{Lem}
\begin{enumerate}
 \item The assignment $M\mapsto M/C_2(M)$ defines a right exact functor
from $\KL_k$ to $\overline{\HC}$.
\item The assignment $M\mapsto \Zhu(M)$ defines a right exact functor
from $\KL_k$ to ${\HC}$.
\end{enumerate}
\end{Lem}

Let $\KL^{\Delta}_k$
be the full subcategory of $\KL$
consisting of modules 
which admit a finite filtration
$0=M_0\subset M_1\subset \dots M_r=M$
such that
$M_i/M_{i+1}\cong M_{\Vg{k}}(E)$ for some finite-dimensional
representation $E_i$ 
for each $i$.
Note that
 the adjoint module  $\Vg{k}$ is an object of $ \KL_k^{\Delta}$
and  that
$M\in \KL_k$
belongs to
 $\KL^{\Delta}_k$ if and only if
it is a
 free $U(\fing[t\inv]t\inv)$-module of finite rank.

 \begin{Lem}\label{Lem:Zhu's-bimodules-of-objects-of-KL}
\begin{enumerate}
\item Let $M$ be 
  an object of $\KL^{\Delta}_k$.
Then 
$\pi_M: M/C_2(M)\ra \gr_F\Zhu(M)$  
%over $\C[\fing^*]$ defined in Lemma
%  \ref{Lem:surjective-hom-ZHu's-bimodule}
is an isomorphism.
\item Let $0\ra M_1\ra M_2\ra M_3\ra 0$ be an exact sequence
in $\KL^{\Delta}_k$.
Then the induced sequence
$0\ra \Zhu(M_1)\ra \Zhu(M_2)\ra \Zhu(M_3)\ra 0$ is exact as well.
\item Let  $M$  be a finitely generated object of $\KL_k$.
Then $\Zhu(M)$ is finitely generated as a left (or a right) $U(\fing)$-module.
\end{enumerate}
 \end{Lem}
 \begin{proof}
(i)
Let $F_{\bullet}O(M)$ be the 
filtration
of $O(M)$ induced 
by the filtration 
$\{M_{\leq p}\}$ of $M$,
% $O_p(M)=O(M)\cap M_{\leq p}$.
$\gr_F O(M)=\bigoplus_{p}F_p O(M)/F_{p-1/2}O(M)$.
%Since
%$O(M)=\haru_{\C[K]}\{(xt^{-1})\circ m| x\in \fing\}$ (see \cite{FreZhu92})
The freeness of $M$ over $U(\fing[t\inv]t\inv)$
implies that 
$a_{(-2)}m\ne 0$  for  any nonzero elements $a\in \Vg{k}$,
$m\in M$.
Hence 
$\gr_F O(M)=C_2(M)\subset M=\gr_F M$ and 
the assertion follows.
(ii)
It is sufficient to show that
the
induced sequence
\begin{align}\label{eq:exact-of-Zhu-bimod}
0\ra \gr_F\Zhu(M_1 )\ra \gr_F \Zhu(M_2 )\ra \gr_F \Zhu(M_3 )\ra 0
\end{align}
is exact.
Since
$0\ra M_1\ra M_2\ra M_3\ra 0$ is  an exact sequence
of free $U(\fing[t\inv]t\inv)$-modules 
it induces an exact sequence
\begin{align*}
0\ra M_1/C_2(M_1)\ra M_2/C_2(M_2)\ra  M_3/C_2(M_3)\ra 0
\end{align*}
by  \eqref{eq:c2-of-affine-module}.
By (i),
this prove the exactness of
(\ref{eq:exact-of-Zhu-bimod}).
(iii)
Since it
is finitely generated, 
$M$
is a quotient of an object of $\KL^{\Delta}$.
By the right exactness of the functor $\Zhu(?)$ it is enough to show
the assertion for
objects of $\KL^{\Delta}$.
By (ii) 
it then  suffices to show the assertion for the modules of the form
$M=M_{\Vg{k}}(E)$.
  But this follows from \cite[Theorem 3.2.1]{FreZhu92}.
 \end{proof}

Let $\{e,f,h\}$ be the $\mf{sl}_2$-triple  defined in Section
\ref{sectopn:classicalBRST}.
In the definition of $W$-algebras $\Wg{k}$ below
we shift the conformal vector $\omega_\fing$ of $\Vg{k}$ to the
conformal vector
\begin{align}
 \omega_{\fing,h}
=\omega_{\fing}+\frac{1}{2}(ht^{-2})\1
\label{eq:shifted-cv}
\end{align}
to give a 
 well-defined conformal vector 
of $\Wg{k}$.
We will identify
Frenkel-Zhu's bimodules 
of $M\in \KL_k$ with respect to 
$ \omega_{\fing,h}$
with 
Frenkel-Zhu's bimodules 
with respect to 
$ \omega_{\fing}$
through Proposition 
\ref{Pro:change-of-grading-does-not-change-the-grading}
and denote both of them  by $\Zhu(M)$.

\section{$W$-algebras and Poisson modules over Slodowy slices}
\label{section:W-algebras}
For a $\Vg{k}$-module $M$,
let
$(C^{\ch}(M),Q_{(0)})$
 be the BRST complex 
of the (generalized) {\em quantized Drinfeld-Sokolov reduction}
associated with $(\fing,f)$ 
defined in \cite{FF90, KacRoaWak03}.
We have
\begin{align*}
 C^{\ch}(M)=M\* \D^{\ch}\* \Lamsemi{\bullet},
\end{align*}
where
$\D^{\ch}$ is the $\beta\gamma$-system of rank $\frac{1}{2}
\dim \fing_{1/2}$,
$\Lamsemi{\bullet}$
 is the space of semi-infinite forms associated with
$\fing_{>0}\+\fing_{>0}^*$.
The vertex algebra
$\D^{\ch}$ is freely generated by the fields
$\phi_i(z)$ with $i=1,\dots, \dim \fing_{1/2}$ (corresponding to the
basis $\{x_i\}$ of $\fing_{1/2}$)
satisfying the OPE's
\begin{align*}
 \phi_i(z)\phi_j(w)\sim \frac{\chi([x_i,x_j])}{z-w}.
\end{align*}
The space $\Lamsemi{\bullet}$
of semi-infinite forms
is a vertex superalgebra freely generated by the odd fields
$\psi_1(z),\dots, \psi_{\dim \fing_{>0}}(z)$
(corresponding to the
basis $\{x_i\}$ of $\fing_{>0}$)
and 
$\psi_1^*(z),\dots, \psi_{\dim \fing_{>0}}^*(z)$
(corresponding to the
dual basis $\{x_i^*\}$ of $\fing_{>0}^*$)
satisfying the OPE's
\begin{align*}
 \psi_i(z)\psi_j^*(w)\sim \frac{\delta_{ij}}{z-w},\quad
 \psi_i(z)\psi_j(w)\sim 
 \psi_i^*(z)\psi_j^*(w)\sim 0.
\end{align*}
The differential $Q_{(0)}$ is the zero-mode of the fields
\begin{align*}
 &Q(z)=\sum_{n\in \Z}Q_{(n)}z^{-n-1}\\
&:=\sum_{i=1}^{\dim \fing_{>0}}
(x_i(z)+\phi_i(z))\psi_i^*(z)-\frac{1}{2}\sum_{1\leq i,j,k\leq \dim
 \fing_{>0}}
c_{ij}^k \psi_i^*(z)\psi_j^*(z)\psi_k(w).
\end{align*}
Here we have omitted the tensor product symbol
and have put
$\phi_i(z)=\chi(x_i)$
for $i>\dim \fing_{1/2}$.
(Note that
in the formula of  $Q(z)$ above  there is no need to take the normal ordering
because of the existence of the structure constant $c_{ij}^k$.)

By abuse of notation we denote also by
$\BRS{0}{M}$ the cohomology of the complex 
$(C^{\ch}(M),Q_{(0)})$.

The {\em $W$-algebra associated with $(\fing,f)$} at level $k$
is by definition
\begin{align}
 \W^k(\fing,f)=H^0_f(\Vg{k}).
\end{align}
The space $\W^k(\fing,f)$ inherits the vertex algebra
structure from $C^{\ch}(\Vg{k})$.
The vertex algebra
$\Wg{k}$ is conformal with
the conformal
vector $\omega_{\W}$
defined by
\begin{align*}
 \omega_{\W}=\omega_{\fing,h}+\omega_{\D}+\omega_{\Lamsemi{\bullet}},
\end{align*}
where 
\begin{align*}
& Y(\omega_{\D},z)=\frac{1}{2}\sum_{i=1}^{\dim
 \fing_{1/2}}:\partial_z\phi^i(z)\phi_i(z),\\
&Y(\omega_{\Lamsemi{\bullet}},z)
=-\sum_{i=1}^{\dim \fing_{>0}}m_i :\psi_i^*(z)\partial_z \psi_i(z):
+\sum_{i=1}^{\dim \fing_{>0}}m_i :\partial_z \psi_i^*(z)\psi_i(z):.
\end{align*}
Here
$\phi^i(z)$ is the field of $\D$ corresponding to the vector $x^i\in
\fing_{1/2}$
such that $\chi([x^i,x_j])=\delta_{ij}$,
 $m_i=j$ if $x_i\in \fing_{j}$,
and we have used the state-field correspondence.
Here the conformal vector 
$\omega_\fing$ of $\Vg{k}$ has been shifted to $\omega_{\fing,h}$
so that
$Q_{(0)}\omega_\W=0$.

By definition 
the assignment 
$M\mapsto \BRS{0}{M}$ defines a functor from
$\Vg{k}\on{-Mod}
$ to $\Wg{k}\on{-Mod}$.

\smallskip

For a $\Vg{k}$-module $M$, consider Zhu's $C_2$-module
$C^{\ch}(M)/C_2 C^{\ch}(M)
$  over the Poisson superalgebra
$R_{C^{\ch}(\Vg{k})}$.
Since we have
$Q_{(0)}C_2 C^{\ch}(M)\subset C_2 C^{\ch}(M)$,
$C^{\ch}(M)/C_2 C^{\ch}(M)$
is a quotient complex,
which
is by definition isomorphic to the 
complex
$(\bar C(M/C_2(M)),\ad \bar d)$  studied in Section \ref{sectopn:classicalBRST}.
We have  the obvious map
\begin{align*}
\bar \eta_M: \BRS{0}{M}/C_2\BRS{0}{M}\ra 
\BRS{0}{M/C_2(M)}.
\end{align*}
For the adjoint  module $M=\Vg{k}$,
$\bar\eta_{\Vg{k}}$ gives the isomorphism
\begin{align*}
\eta_{\Vg{k}}:  R_{\Wg{k}}\isomap  \C[\mc{S}_f]
\end{align*}
(\cite{De-Kac06}).
It follows that
$\bar\eta_{M}$ is a homomorphism of Poisson modules over
$\C[\mc{S}_f]$.

 \begin{Th}[\cite{Ara09b}]
\label{Th:Arakawa}
$ $

\begin{enumerate}
\item  We have $\BRS{i}{M}=0$ for $i\ne 0$,
$M\in \KL_k$.
In particular the functor $\KL_k\ra \Wg{k}\on{-Mod}$,
$M\mapsto \BRS{0}{M}$, is exact.
\item
For $M\in \KL_k$,
$\bar \eta_M$ gives the isomorphism
\begin{align*}
H^0_f(M)/C_2(H_f^0(M))\cong H^0_f(M/C_2(M))
\end{align*}
of Poisson modules over $\C[\mc{S}_f]$.

\end{enumerate}
 \end{Th} 

Let $N$ be an ideal of $\Vg{k}$.
By Theorem \ref{Th:Arakawa} (i)
$\BRS{0}{N}$ embeds into  $\Wg{k}$,
and
we have the isomorphism
\begin{align}
 \BRS{0}{\Vg{k}/N}\cong \Wg{k}/\BRS{0}{N}
\label{eq:general-quotient}
\end{align}
of vertex algebras.
In particular,
\begin{align*}
 \BRS{0}{\Vs{k}}\cong \Wg{k}/\BRS{0}{N_k(\fing)}.
\end{align*}

\section{Quantized Drinfeld-Sokolov reduction and Frenkel-Zhu's bimodules associated with  $W$-algebras}
\label{section:Zhu'ss-bimodules-for-W}
For  a $\Vg{k}$-module $M$,
consider the
bimodule
$\Zhu(C^{ch}(M))$
over $\Zhu(C^{\ch}(\Vg{k}))$.
Since we have $Q_{(0)}O(C^{ch}(M))
\subset O(C^{ch}(M))$,
$(\Zhu(C^{ch}(M)), Q_{(0)})$
is a quotient complex,
which 
is isomorphic to the 
complex
$(C(\Zhu(M)),\ad  d)$  studied in Section \ref{sectopn:classicalBRST},
where,
throughout this section,  $\Zhu(M)$ denotes Frenkel-Zhu's bimodule
associated with $M$ with respect to the conformal vector \eqref{eq:shifted-cv}.
Consider 
the map
\begin{align*}
\begin{array}{cccc}
\eta_M:&\Zhu(\BRS{0}{M})&\ra&  H^0_f(\Zhu(M)),\\
 & [c]+O(H^0_f(M))& \mapsto & [c+O(C(M))].
\end{array}
\end{align*}
For the adjoint module $M=\Vg{k}$,
$\eta_{\Vg{k}}$  gives the isomorphism
\begin{align}
\Zhu(\Wg{k})\isomap
 U(\fing,f)
\label{eq:iso-of-Zhuo-of-universals}
\end{align}
of algebras (\cite{Ara07,De-Kac06}, or see  Proposition
\ref{Pro:surjection} (ii) below).
It follows that $\eta_M$ is a homomorphism of $U(\fing,f)$-bimodules.

We can now state
the main result of 
 this section:
\begin{Th}\label{Th:Zhu-algebra-correspondence}
For any object $M$ of $\KL_k$,
$\eta_M$ gives the isomorphism
\begin{align*}
 \Zhu (\BRS{0}{M}))\cong H^0_f(\Zhu(M))
\end{align*}
of $U(\fing,f)$-bimodules.
\end{Th}
\begin{Rem}
 Theorem \ref{Th:Zhu-algebra-correspondence}
holds at the critical level $k=-h\che_{\fing}$ as well
by considering the outer grading 
as in \cite{Ara05,Ara07}.
\end{Rem}

To avoid confusion
 we denote 
 by
$K_{\bullet}\Zhu(M)$  (instead by $F_{\bullet}\Zhu(M)$)
the filtration
of $\Zhu(M)$  with respect to  the grading 
defined by  the conformal vector (\ref{eq:shifted-cv})
for $M\in \KL_k$.
\begin{Lem}\label{Lem;Zhu-p-is-good}
\begin{enumerate}
\item The filtration $K_{\bullet}\Zhu(\Vg{k})$ coincides with
the Kazhdan filtration of $U(\fing)=\Zhu(\Vg{k})$.
 \item  Let $M$ be an object of $\KL_k$,
Then
$K_{\bullet}\Zhu(M)$ is a Kazhdan filtration of $\Zhu(M)$.
It is good 
if  $M$ is finitely generated.

\end{enumerate}\end{Lem}
\begin{proof}
(i) and 
the first assertion of (ii) is easily seen  from the definition.
To see  the second assertion of (ii)
observe that
$M/C_2(M)$ is a finitely generated 
$\C[\fing^*]$-module
for 
a finitely generated object $M$ of $\KL_k$.
Hence so is $\gr_K \Zhu(M)$ by Lemma \ref{Lem:surjective-hom-ZHu's-bimodule}.
\end{proof}
\begin{Pro}\label{Pro:surjection}
$ $

\begin{enumerate}
\item
For an object  $M$  of $\KL_k$,
$\eta_M: \Zhu(\BRS{0}{M})\ra \BRS{0}{\Zhu(M)}$ 
is surjective.
\item
For an object  $M$  of $\KL_k^{\Delta}$,
$\eta_M: \Zhu(\BRS{0}{M})\ra \BRS{0}{\Zhu(M)}$ 
is an isomorphism.
\end{enumerate}\end{Pro}
\begin{proof}
(i)
First,
suppose that 
$M$ is finitely generated.
By Lemma \ref{Lem;Zhu-p-is-good},
 $K_{\bullet}\Zhu(M)$ is a good Kazhdan filtration of $\Zhu(M)$.
Hence we have \begin{align}
\gr_K \BRS{0}{\Zhu(M)}
\cong \BRS{0}{\gr_K \Zhu( M)}
\label{eq:2012:10-1-2}
\end{align}
by Theorem \ref{Th:vanishing-finite-dimensional}.
Here
$
\gr_K H^0_f(\Zhu(M))
$ is the associated graded with respect to the induced filtration
 $K_p H^0_f(\Zhu(M))=\im (H^0_f(K_p\Zhu(M))\ra 
H^0_f(\Zhu(M)))$.
Since
$\eta_M(K_p\Zhu(H^0_f(M)))\subset K_p H^0_f(\Zhu(M))$,
$\eta_M$ induces a homomorphism
\begin{align*}
 \gr_K \eta_M: \gr_K \Zhu(H^0_f(M))\ra \gr_K \BRS{0}{\Zhu(M)}.
%\label{eq:gr-eta}
\end{align*}
It is enough to show that
$\gr \eta_M$ is surjective.

Consider the surjection
\begin{align*}
\pi_M: M/C_2(M)\twoheadrightarrow  \gr_K \Zhu(M).
\end{align*}
Since both $M/C_2(M)$ and $\gr_K\Zhu(M)$ are 
objects of $\overline{\mc{HC}}$,
this induces the surjection
\begin{align*}
 H^0_f(\pi_M):H_f^0(M/C_2(M))
\twoheadrightarrow  H^0_f(\gr_K \Zhu(M))\cong \gr_K \BRS{0}{\Zhu(M)}
\end{align*}
by Theorem \ref{Th:vanishing-Poisson}.
 
Now
we have the following commutative diagram:
\begin{align}
  \begin{CD}
\BRS{0}{M}/C_2(\BRS{0}{M}) @ >\pi_{H^0_f(M)} >> \gr_K \Zhu(\BRS{0}{M})\\
 @ V \bar \eta_M  VV
  @VV\gr \eta_M  V \\
 H_f^0(M/C_2(M))  @>H^0_f(\pi_M) >>     \gr_K \BRS{0}{\Zhu(M)}.
  \end{CD}
\label{eq:commu-diagram}
\end{align}
Since $\bar \eta_M$ is an isomorphism by Theorem \ref{Th:Arakawa} (ii),
it follows that
$\gr \eta_M$ is surjective as required.

Next,
 let $M$ be an arbitrary object of $\KL_k$.
There exists a sequence of
finitely generated objects $M_0\subset M_1\subset
 M_2\subset
\dots $ in $\KL_k$
such that
$M=\bigcup_i M_i$.
Since
(co)homology functor  commutes with injective limits,
$\Zhu(M)=\lim\limits_{\longrightarrow
\atop  i}\Zhu(M_i)$,
$\BRS{0}{M}
=\lim\limits_{\longrightarrow \atop i}\BRS{0}{M_i}$,
$\Zhu(H^0_f(M))=\lim\limits_{\longrightarrow \atop i}\Zhu(\BRS{0}{M_i})$,
and 
$H^0_f(\Zhu(M))=\lim\limits_{\longrightarrow
\atop  i}H^0_f(\Zhu(M_i))$.
This proves the  assertion.

(ii)
 By Lemma \ref{Lem:Zhu's-bimodules-of-objects-of-KL} (i)
$H^0_f(\pi_M)$ is an isomorphism.
Hence the commutativity of (\ref{eq:commu-diagram})
implies that 
$\pi_{\BRS{0}{M}}$ and
$\gr \eta_M$ are  isomorphisms,
and hence, so is $\eta_M$.
\end{proof}
\begin{proof}[Proof of Theorem \ref{Th:Zhu-algebra-correspondence}]
As in the proof of Proposition   \ref{Pro:surjection}
it is sufficient to show the case that
$M$ is finitely generated.
Then
there exists an
exact sequence
\begin{align}
 0\ra N\ra V\ra M\ra 0
\label{eq:exact-seq-started}
\end{align}
in the category $\KL_k$
with $V\in \KL^{\Delta}_k$.
By the right exactness of the functor
$\Zhu(?)$
this yields an exact sequence
\begin{align*}
 \Zhu(N)\ra  \Zhu(V)\ra \Zhu(M)\ra 0
\end{align*}
in the category $\mc{HC}$.
Applying the 
exact functor $H_f^0(?):
\mc{HC}\ra U(\fing,f)\on{-biMod}$ (Theorem \ref{Th:vanishing-finite-dimensional})
to the above  sequence
we obtain an exact sequence
\begin{align*}
H^0_f( \Zhu(N))\ra  H^0_f(\Zhu(V))\ra H^0_f(\Zhu(M))\ra 0.
\end{align*}
 
On the other hand
by applying the exact functor $H^0_f(?):
\KL_k\ra \Wg{k}\on{-Mod}$
(Theorem \ref{Th:Arakawa}) to (\ref{eq:exact-seq-started})
we obtain  the exact sequence
\begin{align*}
 0 \ra H^0_f(N)\ra H^0_f(V)\ra H^0_f(M)\ra 0.
\end{align*}
This yields   an exact sequence
\begin{align}
 \Zhu(H^0_f(N))\ra \Zhu(H^0_f(V))\ra \Zhu(H^0_f(M))\ra 0.
\end{align}

Now we have the following
commutative diagram:
\begin{align}
 \begin{CD}
 \Zhu(H^0_f(N))@>>>\Zhu(H^0_f(V))@>>> \Zhu(H^0_f(M))@>>> 0\\
@ V\eta_N VV  @ V{\eta_V}VV  @V\eta_M VV\\
H^0_f( \Zhu(N))@>>> H^0_f(\Zhu(V))@>>> H^0_f(\Zhu(M))@>>> 0.
 \end{CD}&
\end{align}
By Proposition \ref{Pro:surjection}
 $\eta_N$ and $\eta_M$ are surjective
and 
$\eta_V$ is an isomorphism.
As the horizontal sequences are exact
it follows that
$\eta_M$ is an isomorphism.
This completes the proof.
\end{proof}

For an  ideal 
 $N$ of $\Vg{k}$,
let
$J_N$ denote the image of $\Zhu(N)$ in $\Zhu(\Vg{k})=U(\fing)$,
so that
\begin{align}
 A(\Vg{k}/N)=U(\fing)/J_N.
\label{eq:zhu-of-quotient}
\end{align}
Note that
$\BRS{0}{\Vg{k}/N}$ is a quotient vertex  algebra
of $\Wg{k}$ 
 provided it is nonzero (see \eqref{eq:general-quotient}).
\begin{Th}\label{TH:Zhu-algebra-of-quotients}
For any ideal  $N$  of $\Vg{k}$,
we have the isomorphism of algebras
\begin{align*}
 \Zhu(\BRS{0}{\Vg{k}/N})&\cong U(\fing,f)/\BRS{0}{J_N}.
\end{align*}
\end{Th}
 \begin{proof}
Set $L=\Vg{k}/N$.
By Theorem \ref{Th:Zhu-algebra-correspondence},
\begin{align*}
 \Zhu(\BRS{0}{L})\cong \BRS{0}{\Zhu(L)},
\end{align*}
and
by Theorem \ref{Th:vanishing-finite-dimensional}
the exact sequence
$0\ra J_N\ra U(\fing)\ra \Zhu(L)\ra 0$
induces the exact sequence
\begin{align*}
 0\ra \BRS{0}{J_N}\ra U(\fing,f)\ra \BRS{0}{\Zhu(L)}\ra 0.
\end{align*}
This completes the proof.
 \end{proof}

The following assertion follows 
immediately from Theorems \ref{eq:equivalence-of-categoroes}
and \ref{TH:Zhu-algebra-of-quotients}.
\begin{Th}\label{Th:equivalence}
For any ideal $N$ of
$\Vg{k}$ we have
 the equivalence of categories
\begin{align*}
\mc{C}^{J_N}\isomap \Zhu(\BRS{0}{\Vg{k}/N})\on{-Mod},
\quad M\mapsto \on{Wh}^{\finm}(M).
\end{align*}
A quasi-inverse functor is given by $E\mapsto Y\*_{U(\fing,f)}E$.
\end{Th}

\section{Varieties associated with Zhu's algebras
of admissible affine vertex algebras }
\label{section:varieties}
Let 
$\fing=\finn_-\+\finh\+\finn$ be a triangular decomposition
of $\fing$ with Cartan subalgebra $\finh$,
$\Delta$ the set of  roots of $\fing$,
$\Delta_+$ the set of positive roots of $\fing$,
$W$ the Weyl group of $\fing$,
$Q\che\subset \finh$ the coroot lattice of $\fing$,
$P\che\subset \finh$ the coweight lattice of $\fing$,
$\rho$  the half sum of positive roots of $\fing$,
$\rho\che$
 the half sum of positive coroots of $\fing$.
For $\lam\in \dual{\finh}$,
let 
$\Ms{ \lam}$ be the Verma module
of $\fing$ with highest weight $ \lam\in \dual{\finh}$,
 $\Ls{\lam}$
 the unique simple quotient of $\Ms{\lam}$.

Let
$\affh=\finh\+ \C K$  be the Cartan subalgebra 
of $\affg$,
$\dual{\affh}=\finh^* \+\C \Lam_0$ the dual of $\affh$,
where $\Lam_0(K)=1$,
$\Lam_0(\finh)=0$.
Let 
$\wh{\Delta}^{re}$ be the set of real roots in the dual
$\dual{\tilde{\finh}}$
of the extended Cartan subalgebra $\tilde{\finh}$ of $\affg$,
$\wh{\Delta}^{re}_+$ the set of positive real roots,
$\affW=W\ltimes Q\che$ the Weyl group of $\affg$,
$\eW=W\ltimes P\che$ the extended Weyl group of $\affg$,
$\wh{\rho}=\rho+h\che \Lam_0$.
For $\lam \in \dual{\affh}$,
let $\wh{\Delta}(\lam)=\{\alpha\in \wh{\Delta}^{re}| 
\bra \lam+\wh\rho,\alpha\che\ket \in \Z\}$, the set of integral roots of
$\lam$,
$\affW(\lam)=\bra s_{\alpha}|  \alpha\in \wh{\Delta}(\lam)\ket
\subset \affW$ the
integral
Weyl group of $\lam$,
where $s_{\alpha}$ is the reflection 
with respect to $\alpha$.
Denote by  $\bar \lam$  the restriction of $\lam\in \dual{\affh}$
to $\finh$.

Set
\begin{align*}
 \affh_k^*=\{\lam\in \dual{\affh}| 
\lam(K)=k\},
\end{align*}
the set of weights of $\affg$ of level $k$.
For $\lam\in \dual{\affh}_k$,
let $L(\lam)$ be the irreducible representation of 
$\affg$ with highest weight $\lam$.
Clearly, 
$L(\lam)$ is irreducible  as a $\Vg{k}$-module.

A weight $\lam\in \dual{\affh}$ is called 
{\em admissible} if 
(1) $\lam$ is regular dominant, that is,
$\bra \lam+\wh\rho,\alpha\che\ket\not \in \{0,-1,-2,-3,\dots\}$
for all $\alpha\in \wh{\Delta}^{re}_+$,
and (2) $\Q\wh{\Delta}(\lam)=\Q\wh{\Delta}^{re}$.
The admissible weights of
$\affg$ were classified in \cite{KacWak89}.
The module $L(\lam)$ is called  admissible
if  $\lam$ is admissible.
Admissible representations
are (conjecturally  all) modular invariant representations of $\affg$
(\cite{KacWak89}).

A number 
 $k$ is called  {\em admissible  for $\affg$}
if $k\Lam_0$ is an admissible weight.
By  \cite[Proposition 1.2]{KacWak08},
$k$ is an admissible number for $\affg$ if and only if
\begin{align}
 k+h\che=\frac{p}{q},
\quad p,q\in \N,\quad (p,q)=1,\quad
 p\geq 
\begin{cases}
h_{\fing}\che&\text{if }(r\che,q)=1,\\
h_{\fing}&\text{if }(r\che,q)=r\che.
\end{cases}
\label{eq:admissible number}
\end{align}
%Here $h$,
%$h\che$,
%$r\che$ are the Coxeter number,
%the dual Coxeter number,
%the lacing number of $\fing$,
%respectively.
A number $k$
of the form
(\ref{eq:admissible number})
is called an {\em admissible number with denominator $q$.}

For an admissible number $k$ of $\affg$,
let 
$Pr^k$ be the set of admissible weights
$\lam$
of level $k$ such that
$\wh{\Delta}(\lam)\cong \wh{\Delta}(k\Lam_0)$ as root systems.
%The following assertion was 
%and proved in \cite{A12-2}.
 \begin{Th}[\cite{A12-2}]
\label{Th:classification-of-simple-modules-over-admissible-affine}
Let $k$ be an admissible number for $\affg$,
$\lam\in \dual{\affh}_k$.
Then $\Irr{\lam}$ is a 
module over the vertex algebra $\Irr{k\lam_0}$ if and only if
$\lam\in Pr^k$.
In particular the vertex operator algebra $L(k\Lam_0)$  is rational in the category $\BGG$
of $\affg$ as conjectured in
\cite{AdaMil95}.
 \end{Th}

By Zhu's theorem,
the first statement of
Theorem \ref{Th:classification-of-simple-modules-over-admissible-affine}
is equivalent to that
$\Ls{\lam}$  with $\lam\in \dual{\finh}$ is a module over $\Zhu(\Vs{k})$
if and only if $ \lam+k\Lam_0\in Pr^k$.
On the other hand  by Duflo's theorem \cite{Duf77}
any primitive ideal of $U(\fing)$ is the annihilating ideal
of some irreducible highest weight module $\Ls{\lam}$.
Hence
Theorem \ref{Th:classification-of-simple-modules-over-admissible-affine}
implies the following.
\begin{Co}\label{Co:admisiblle}
Let $k$ be an admissible number
for $\affg$.
A simple $U(\fing)$-module
$M$
 is an 
$\Zhu(\Vs{k})$-module if and only if
$\on{Ann}_{U(\fing)}M=\on{Ann}_{U(\fing)}{L}_{\fing}({\bar \lam})$
for some 
$\lam\in Pr^k$.
\end{Co}

Let $k$ be   an admissible number  for $\affg$.
We shall  determine
\begin{align*}
 \Var\Zhu(\Vs{k}):=\Specm (\gr_F\Zhu(\Vs{k}))
(\cong \Specm (\gr_K\Zhu(\Vs{k}))),
\end{align*}
which is a $G$-invariant, conic, Poisson subvariety  of $\fing^*$.

Recall \cite{Ara12} 
that the associated variety
$X_V$  %$X_{\Vs{k}}$
of a finitely strongly generated vertex algebra $V$ %$\Vs{k }$
is defined as
\begin{align*}
&X_{V}
=\Specm(R_{V}).
%\label{eq:variety-of-Zhu-algebra}
\end{align*}
Note that 
$V$ is $C_2$-cofinite if and only if $X_V$ is zero-dimensional.

By  
\eqref{eq:surjective-to-Zhu},
$\Var \Zhu(\Vs{k})$ is a subvariety of $X_{\Vs{k}}$,
which is also a $G$-invariant, conic, Poisson subvariety  of $\fing^*$.

Let us identify $\fing^*$ with $\fing $ through  $\nu$,
and 
let  $\mc{N}\subset \fing^*= \fing$ be  the nilpotent cone.

By a conjecture of Feigin and Frenkel proved in \cite{Ara09b}
we have 
\begin{align*}
X_{\Vs{k}}\subset \Nil\quad\text{for an admissible number $k$ for $\affg$.}
\end{align*}
In fact the following holds:
\begin{Th}[\cite{Ara09b}]\label{Th:vareity-of-admissible-affine}
Let $k$ be an 
admissible number for $\affg$.
Then
$X_{\Vs{k}}$ is an irreducible subvariety of $\Nil$
which depends only on the denominator $q$ of $k$, that is,
 there exist a nilpotent element $f_q$ of $\fing$
such that
\begin{align*}
X_{\Irr{k\Lam_0}}=\overline{\Ad G.f_q}.
\end{align*}
More explicitly, we have
\begin{align*}
 X_{\Vs{k}}
= \begin{cases}
  \{x\in \fing|  (\ad x)^{2q}=0\}& \text{if }(q,r\che)=1,\\
\{x\in \fing|  \pi_{\theta_s}(x)^{2q/r\che}=0\}
& \text{if }(q,r\che)=r\che,
 \end{cases}&
\end{align*}
where
$\theta_s$ is the highest short root of $\fing$
and
$\pi_{\theta_s}:\fing\ra \End_{\C}(L_{\fing}(\theta_s))$ 
is the finite dimensional irreducible representation of $\fing$
with highest weight $\theta_s$.
 \end{Th}

Theorem \ref{Th:vareity-of-admissible-affine}
has the following important consequence \cite{Ara09b}:
By Theorems 
\ref{Th:vanishing-Poisson} and \ref{Th:Arakawa}
we have
\begin{align}
 X_{\BRS{0}{\Vs{k}}}\cong X_{\Vs{k}}\cap \mc{S}_f.
\label{eq:associated-vareity-of-W}
\end{align}
Hence the transversality of $\mc{S}_f$ with $G$-orbits 
(see \cite{GanGin02})
 implies the following:
 \begin{Th}[\cite{Ara09b}]\label{Th:C2-cofiniteness}
Let $k$ be an admissible number with denominator $q$.
Then
the vertex algebra
$H_{f_q}^0(\Irr{k\Lam_0})$ 
is a non-zero $C_2$-cofinite
quotient of $\Wg{k}$.
 \end{Th}

Now we are in a position to state the main result of this section.
\begin{Th}\label{Th:variety-of-Zhu's-algebra}
 Let $k$ be an admissible number for $\affg$
with denominator $q$ .
We have  an isomorphism of affine varieties
\begin{align*}
\Var \Zhu(\Irr{k\Lam_0})
\cong  X_{\Vs{k}}.
\end{align*}
\end{Th}
\begin{proof}
By Theorem \ref{Th:vareity-of-admissible-affine},
it is sufficient to show the following assertion.
\begin{Pro}\label{Pro:has-nilp-ortbis}
 Let $f$ be any nilpotent element of $\fing$,
and let $k$ be any complex number.
The following conditions are equivalent:
\begin{enumerate}
 \item $X_{L(k\Lam_0)}\supset \overline{\Ad G.f}$.
\item $\Var (A(L(k\Lam_0)))
\supset \overline{\Ad G.f}$.
\end{enumerate}
\end{Pro}
\begin{proof}
 Clearly (ii) implies (i) as
$\Var A(L(k\Lam_0)))\subset X_{L(k\Lam_0)}$.
Conversely,
suppose that
$X_{L(k\Lam_0)}\supset \overline{\Ad G.f}$.
Since $\Var \Zhu(\Vs{k})$ is  $G$-invariant
and
closed 
it is sufficient to show that
the point $f\in \fing=\fing^*$ is contained in
$  \Var \Zhu(\Vs{k})$.
By \eqref{eq:associated-vareity-of-W},
$X_{H^0_f(L(k\Lam_0))}$ contains  $f$, and hence,
$H^0_f(L(k\Lam_0))
$ is nonzero.
It follows that
$A(H^0_f(L(k\Lam_0)))=H^0_f(A(L(k\Lam_0)))$ is nonzero as well. 
Since
$\Var H^0_f(A(L(k\Lam_0)))
=\Var A(L(k\Lam_0))
\cap \mc{S}_f$
by Theorem \ref{Th:vanishing-finite-dimensional} (i),
$\Var A(L(k\Lam_0))$ intersects
$\mc{S}_f$ non-trivially.
As
$\Var H^0_f(A(L(k\Lam_0)))$ is invariant 
under 
the natural $\C^*$-action on $\mc{S}_{f}$ 
which is contracting to $f$
(see \cite{Gin08}),
$\Var \Zhu(\Vs{k})$  must contain the point $f$ as required.
\end{proof}
\end{proof}
 \begin{Conj}
For a finitely strongly generated simple vertex operator algebra
$V$
of  CFT type
we have
$ \Var \Zhu(V)(:=\Specm \gr_F(\Zhu(V)))\cong  X_V$.
 \end{Conj}

Note that
Conjecture 1 in particular implies the widely believed fact that
a finitely strongly generated rational vertex operator algebra
of CFT type must be $C_2$-cofinite.

\section{Proof of Main Theorem}
\label{section:proof}
In this section we let $f=f_{prin}$,
a principal nilpotent element of $\fing$,
 \begin{align*}
&\W^k(\fing)=\W^k(\fing,f_{prin})=H^0_{f_{prin}}(\Vg{k}),\\
\text{and }&\W_k(\fing)=\text{ the unique simple quotient of }\W^k(\fing)
 \end{align*}
as in Introduction.
The vertex algebra
$\W^k(\fing)$ is
$\Z_{\geq 0}$-graded
by $L_0$,
where
\begin{align*}
 Y(\omega_{\W},z)=\sum_{n\in \Z}L_n z^{-n-2}.
\end{align*}
The central charge 
$c(k)$ of $\W^k(\fing)$ is 
 given in Introduction.
We have  
the isomorphisms
\begin{align*}
& \C[\fing^*]^{G}\isomap  \C[\mc{S}_f]=H^0(\bar C(\C[\fing^*]),\ad \bar d)\cong
R_{\W^k(\fing)},
\quad p\mapsto p\* 1,\\
& \mc{Z}(\fing)\isomap  U(\fing,f_{prin})=H^0(C(U(\fing)),\ad d)\cong
\Zhu(\W^k(\fing)),
\quad z\mapsto z\* 1
\end{align*}
(\cite{Kos78}, see also \cite{Ara07}),
where $\mc{Z}(\fing)$ denotes the
center of $U(\fing)$.
We will identity $\Zhu(\W^k(\fing))$ with $\mc{Z}(\fing)$
through the above isomorphism.

For a central character $\gamma:\mc{Z}(\fing)\ra \C$,
let $\C_{\gamma}$ be the one-dimensional representation of $\mc{Z}(\fing)$
defined by $\gamma$.
Put
\begin{align*}
 \VermaW{\gamma}=M_{\W^k(\fing)}(\C_{\gamma}),
\quad  \IrrW{\gamma}=L_{\W^k(\fing)}(\C_{\gamma})
\end{align*}
(see section \ref{section:Zhu-algberas}).
We have
\begin{align*}
\on{Irr}(\W^k(\fing))=\{\IrrW{\gamma_{\lam}}| 
\lam\in \dual{\finh}/W-\rho \},
\end{align*}
where
 $\gamma_{ \lam}: \mc{Z}(\fing)\ra \C$ is
the
evaluation at 
$\Ms{ \lam}$.
Note that
\begin{align*}
\W_k(\fing)\cong \IrrW{\gamma_{-(k+h\che_{\fing})\rho\che}},
 \end{align*}
see \cite[5.4]{Ara07}.

 \begin{Th}\label{Th:classification.general}
Let
$N$ be an ideal of $\Vg{k}$
and 
suppose that
$H_{f_{prin}}^0(\Vg{k}/N)\ne 0$,
so that
$H_{f_{prin}}^0(\Vg{k}/N)$
is a quotient vertex algebra 
of $\W^k(\fing)$ (see \eqref{eq:general-quotient}).
We have
\begin{align*}
\on{Irr}(H_{f_{prin}}^0(\Vg{k}/N))=
 \{\IrrW{\gamma}|  U(\fing)\ker\gamma\supset J_N
\}
\end{align*}
(Here $J_N$ is defined in Section \ref{section:Zhu'ss-bimodules-for-W},
see \eqref{eq:zhu-of-quotient}).
 \end{Th}
  \begin{proof}
Recall 
Skryabin's equivalence for $f=f_{prin}$ in Section 
\ref{section:finite W-algebras}:
\begin{align*}
\mc{Z}(\fing)\on{-Mod}\isomap \mc{C},\quad
E\mapsto Y\*_{\mc{Z}(\fing)}E,
\end{align*}
which goes back to Kostant \cite{Kos78}.
In particular,
$\{Y_{\gamma}|  \gamma\in \finh^*/W-\rho\}$ gives the complete set of isomorphism
classes of simple object of $\mc{C}$,
where $Y_{\gamma}=Y\*_{\mc{Z}(\fing)}\C_{\gamma}$.
We have  \cite{Kos78}
\begin{align*}
\Ann_{U(\fing)} Y_{\gamma}=U(\fing)\ker \gamma.
\end{align*}
Therefore
$Y_{\gamma}$ is annihilated by $J_N$ if and only if
$J_N\subset U(\fing)\ker \gamma$.
In other words
$\{Y_{\gamma}|  U(\fing)\ker \gamma\supset J_N
\}$ gives the complete  set of
isomorphism classes of
simple objects of $\mc{C}^{J_N}$.
By Theorem \ref{Th:equivalence}
this is equivalent to that
\begin{align*}
\on{Irr}(A(\BRS{0}{\Vg{k}/N}))
=\{\C_{\gamma}|  U(\fing)\ker\gamma\supset J_N
\}.
\end{align*}
This completes the proof.
  \end{proof}

Recall that
$X_{\Vs{k}}\subset \Nil$
for
an 
admissible number $k$
for $\affg$
  (Theorem \ref{Th:vareity-of-admissible-affine}).
An admissible  number $k$ 
is called {\em non-degenerate}
if 
\begin{align*}
X_{\Irr{k\Lam_0}}=\Nil=
\overline{\Ad G.f_{prin}}.
\end{align*}
From Theorem
  \ref{Th:vareity-of-admissible-affine}
and the fact that
\begin{align}
(\theta|\rho\che)=h_{\fing}-1,\quad (\theta_s|\rho\che)=h_{{}^L{\fing}}\che-1,
\label{eq:hche-dualhche}
\end{align}
where $\theta$ is the highest root of $\fing$,
it follows that
an admissible number  $k$ is
non-degenerate if and only if 
$k$ satisfies 
\begin{align*}
q\geq \begin{cases}
	      h_{\fing}&\text{if }(q,r\che)=1,\\
r\che  h_{{}^L\fing}
\che&\text{if }(q,r\che)=r\che,
	     \end{cases}
\end{align*}
where
$q$ is the denominator of $k$,
that is,
$k$ is of the form (\ref{eq:form-of-k}).

 \begin{Th}\label{Th:simple=image-of-simple}
Let $k$ be an admissible number for $\affg$.
Then
$H_{f_{prin}}^0(\Vs{k})
\ne 0 $
if and only if 
$k$ is non-degenerate.
If this is the case
then
\begin{align*}
H_{f_{prin}}^0(\Vs{k})\cong \W_k(\fing).
\end{align*}
Moreover,
 $\W_k(\fing)$ is $C_2$-cofinite.
 \end{Th}
\begin{proof}
 The fact that
$H_{f_{prin}}^0(\Vs{k})
\cong \W_k(\fing)$ for a non-degenerate
admissible number
$k$ was 
 proved in \cite[Theorem 9.1.4]{Ara07}.
The rest of the assertion is the special case of Theorem \ref{Th:C2-cofiniteness}.
\end{proof}

Let 
\begin{align*}
Pr^k_{\mathit{non-deg}}=\{\lam\in Pr^k|  \bra \lam,\alpha\che\ket \not \in \Z
\text{ for all }\alpha\in \Delta\},
\end{align*}
 the set of {\em non-degenerate admissible weights}
 \cite[Lemma 1.5]{FKW92} of
level
$k$.
It is known 
\cite{FKW92} that
$Pr_{non-deg}^k$ is non-empty
if and only if $k$ is non-degenerate.
Put
\begin{align*}
Pr^k_{\W}=\{\gamma_{\bar \lam}|  \lam\in Pr_{\mathit{non-deg}}^k\}.
\end{align*}
Then
$\sharp Pr^k_{\W}=\sharp Pr^k_{non-deg}/\sharp W$
since
$W$ acts on $Pr^k_{non-deg}$ freely (by the dot action).

The irreducible representations 
$\{\IrrW{\gamma}|  \gamma\in Pr^k_{W}\}$ are  called
 {\em minimal series representations} of $\W^k(\fing)$.
 In
\cite{Ara07} we have verified 
the conjectural character formula of 
minimal series representations
of
$\W_k(\fing)$ given by
Frenkel-Kac-Wakimoto  \cite{FKW92}.
(In fact the main result of \cite{Ara07}
gives the character of all $\IrrW{\gamma}$,
see Theorem \ref{Th:character-of-W-modules}
and Corollary \ref{Co:multiplicity-formula} below.)

\begin{Rem}
The module
 $\IrrW{\gamma}$ with $\gamma\in Pr^k_{\W}$ admits a
two-sided resolution in terms of 
free field realizations  \cite{A-BGG}.
However we do not need this result.
\end{Rem}

 \begin{Th}\label{Th:classification-of-simlple-modules-of-minimal-W}
Let $k$ be  a non-degenerate admissible number  for $\affg$,
$\gamma$ a central character of $\mc{Z}(\fing)$.
Then $\IrrW{\gamma}$ is a 
module over $\W_k(\fing)$ if and only if it is a minimal series
  representation
of $\W^k(\fing)$,
that is,
\begin{align*}
\on{Irr}(\W_k(\fing))=\{\IrrW{\gamma}|\gamma \in Pr_{\W}^k\}.
\end{align*}
 \end{Th}
\begin{proof}
Set $J_k=J_{N_k(\fing)}$,
so that
\begin{align*}
 A(L(k\Lam_0))=U(\fing)/J_k.
\end{align*}
By
Theorem \ref{Th:simple=image-of-simple},
we have $\W_k(\fing)=H^0_{f_{prin}}(\Vg{k}/N_k(\fing))$.
Hence
 Theorem 
\ref{Th:classification.general}
gives that
\begin{align*}
 \on{Irr}(\W_k(\fing))= \{\IrrW{\gamma}|  U(\fing)\ker \gamma\supset J_k\}.
\end{align*}

Now recall that $\bar \lam\in \dual{\finh}$ is called 
{\em anti-dominant} if $\bra\bar \lam+\rho,\alpha\che\ket\not\in \N$
for all $\alpha\in \Delta_+$.
Clearly,
for any central character $\gamma:\mc{Z}(\fing)\ra \C$
there exists an anti-dominant $\bar \lam\in \dual{\finh}$
such that $\gamma=\gamma_{\bar \lam}$.
It is well-known that   $\Ls{\bar \lam}=M_{\fing}(\bar \lam)$
for an anti-dominant $\bar \lam$
and that
\begin{align*}
 \Ann_{U(\fing)}M_{\fing}(\bar \lam)=U(\fing)\ker \chi_{\bar \lam}.
%\label{eq:annihikator-of-Verma}
\end{align*}

We have 
\begin{align*}
 &\{\IrrW{\gamma}|  U(\fing)\ker \gamma
\supset J_k\}\\
=&
\{\IrrW{\gamma_{\bar \lam}}|  \bar \lam\in \dual{\finh}, 
\text{ $\bar \lam$ is anti-dominant, }
 \Ann_{U(\fing)}\Ls{\bar \lam}\supset J_k\}\quad\text{(by the above)}
\\
=&
\{\IrrW{\gamma_{\bar \lam}}|  \bar \lam\in \dual{\finh},  \text{ $\bar \lam$ is anti-dominant, 
$\Ls{\bar \lam}$ is an 
$\Zhu(\Vs{k})$-module}\}\\
=&
\{\IrrW{\gamma_{\bar \lam}}|  \lam\in \dual{\affh},  \text{ $\bar \lam$ is anti-dominant, 
$L(\lam)$ is an 
$\Vs{k}$-module}\}\\
=&\{\IrrW{\gamma_{\bar \lam}}|  \lam\in Pr^k,\ \text{$\bar \lam$ is
 anti-dominant}\}\quad\text{(by Theorem
 \ref{Th:classification-of-simple-modules-over-admissible-affine})}\\
=&\{\IrrW{\gamma_{\bar \lam}}|  \lam\in Pr^k_{non-deg}\}
=\{\IrrW{\gamma}|\gamma \in Pr_{\W}^k\}.
\end{align*}
This completes the proof.
\end{proof}
 
\begin{Th}\label{Th:semi-simplicity-of-Zhu}
For  a non-degenerate admissible number $k$ for $\affg$,
Zhu's algebra
 $\Zhu(\W_k(\fing))$ is semisimple.
\end{Th}
In order to prove Theorem \ref{Th:semi-simplicity-of-Zhu},
we consider the Lie algebra
{\em homology} functor
\begin{align*}
\fing\Mod \ra  \mc{Z}(\fing)\Mod, 
\quad M\mapsto H_0(\finn_-,M).
\end{align*}
Since $\Ms{\lam}$ is free over $U(\finn_-)$,
\begin{align}
H_i(\finn_-, \Ms{\lam})\cong \begin{cases}
					\C_{\gamma_{\lam}}&\text{for
					}i=0,\\0&\text{for }
i>0.			       \end{cases}
\label{eq:whitttaker}
\end{align}
\begin{Lem}\label{Lem:extension}
Let $\lam\in \finh^*$  be regular, 
that is, $\bra \lam+\rho,\alpha\che\ket \ne 0$
for all $\alpha\in \Delta$.
Then
for an exact sequence
$0\ra \C_{\gamma_{\lam}}\overset{\phi_1}{\ra} E\overset{\phi_2}{\ra}
 \C_{\gamma_{ \lam}}\ra 0$
of $\mc{Z}(\fing)$-modules,
there exists 
an exact sequence
$0\ra  M_{\fing}( \lam)\ra N\ra  M_{\fing}( \lam)\ra 0$
of $\fing$-modules such 
$E\cong H_0(\finn_-, N)$
as $\mc{Z}(\fing)$-modules.
\end{Lem}
 \begin{proof}
Choose
homogeneous generators 
$p_1,\dots,p_{\rank \fing}$ of
of 
$\mc{Z}(\fing)$.
Let 
\begin{align*}
\Upsilon: \mc{Z}(\fing)\isomap S(\finh)^W
\end{align*}
be the Harish-Chandra isomorphism,
so that $z v_{\lam}=\Upsilon(z)(\lam+\rho)v_{\lam}$
for $z\in \mc{Z}(\fing)$,
where $v_{\lam}$ is the highest weight vector of $\Ms{\lam}$.
Set
$v=\phi_1(1)$
and
fix $v'\in E$ such that $\phi_2(v')=1$.
Then there exists $d_1,\dots,d_{\rank \fing}\in \C$
such that
\begin{align*}
p_i v'=\Upsilon(p_i)(\lam+\rho)v'+d_i v.
\end{align*}

Let us identify $S(\finh)$ with
 $\C[\alpha_1\che,\dots,\alpha_{\rank \fing}\che]$.
It is well-known that
\begin{align}
\det (\frac{\partial \Upsilon(p_i)}{\partial \alpha_j\che})_{1\leq i,j
\leq \rank \fing}=C\prod_{\alpha\in
 \Delta_+}\alpha\che,
\label{eq:Jacobian}
\end{align}
where $C$ is some nonzero constant.
The hypothesis on $\lam$
implies that
the value of  (\ref{eq:Jacobian})
at $\lam+\rho$ is non-zero.
It follows that
there exists
some
$\mu\in \finh^*$
such that
\begin{align}
 \Upsilon(p_i)(\lam+t\mu+\rho)=\Upsilon (p_i)(\lam+\rho)+t d_i+O(t^2)
\label{eq:Maria}
\end{align}
for all $i=1,\dots,\rank\fing$.

Let $A=\C[t]$,
$\finh_A=\finh\*_{\C}A$.
Denote by  $A_{\lam+t\mu}$
the $\finh_A$-module that is
a rank one free $A$-module on which
$h\in \finh$ acts as multiplication by the scalar $\lam(h)+t\mu(h)$.
Set
$M=A_{\lam+t\mu}/t^2 A_{\lam+t\mu}$
and view $M$ as an $\finh$-module.
Observe that
$t M\cong \C_{\lam}$ and
we have the exact sequence
\begin{align}
 0\ra t M\ra M\ra \C_{\lam} \ra 0
\label{eq:exact-seq-A}
\end{align}
of $\finh$-modules.
Set 
\begin{align*}
N=U(\fing)\*_{U(\finb)}M,
\end{align*}
where $\finb=\finh\+ \finn$
and $M$ is regarded as a $\finb$-module via the
natural surjection $\finb\ra \finh$.
Applying the induction functor $U(\fing)\*_{U(\finb)}?$
to (\ref{eq:exact-seq-A})
we obtain the exact sequence
\begin{align}
 0\ra M_{\fing}(\lam)\ra N
\ra M_{\fing}(\lam)\ra 0
\label{eq:exact1}
\end{align}
of $\fing$-modules.
Next  applying  the functor $H_0(\finn, ? )$
we get 
 the exact sequence
\begin{align*} 
 0\ra \C_{\gamma_{\lam}}
\ra H_0(\finn_-, N
)\ra\C_{\gamma_{\lam}}
\ra 0
\end{align*}
of $\mc{Z}(\fing)$-modules
by  (\ref{eq:whitttaker}).
By construction,
$H_0(\finn_-, N)\cong E$ as required.
 \end{proof}

\begin{Pro}\label{Pro:induction}
For $\lam\in Pr^k$
we have $
L(\lam)\cong M_{\Vs{k}}(\Ls{\bar \lam})
$
(see \eqref{eq:def-of-Verma}).
\end{Pro}
 \begin{proof}
We have a surjective map
\begin{align*}
M_{\Vg{k}}(\Ls{\bar \lam})=
 U(\affg)\*_{U(\fing[t]\+ \C K)}\Ls{\bar \lam}\twoheadrightarrow
M_{\Vs{k}}(\Ls{\bar \lam})
\end{align*}
of $\affg$-modules.
It follows that
$M_{\Vs{k}}(\Ls{\bar \lam})$
 is an object of
$\BGG$ of $\affg$.
Being a $\Vs{k}$-module,
$M_{\Vs{k}}(\Ls{\bar \lam})$
 decomposes into  a direct sum of admissible representations
by Theorem
  \ref{Th:classification-of-simple-modules-over-admissible-affine}.
Since it is generated by the highest weight vector of $\Ls{\bar \lam}$,
$M_{\Vs{k}}(\Ls{\bar \lam})$
must be isomorphic to
$L(\lam)$.
 \end{proof}

\begin{proof}[Proof of Theorem \ref{Th:semi-simplicity-of-Zhu}]
Since
$\W_k(\fing)=H^0_{f_{prin}}(L(k\Lam_0))$
is $C_2$-cofinite
by Theorem \ref{Th:C2-cofiniteness},
Zhu's algebra $A(\W_k(\fing))$
is finite-dimensional.
Also, we have shown that
 $\on{Irr}(\Zhu(\W_k(\fing)))=\{\C_{\gamma}| \gamma\in Pr^k_{\W}\}$
in Theorem \ref{Th:classification-of-simlple-modules-of-minimal-W}.

Let $\lam \in Pr_{non-deg}^k$,
and let 
\begin{align}
 0\ra \C_{\gamma_{\bar \lam}}\ra E\ra \C_{\gamma_{\bar \lam}}\ra 0
\label{eq:non-spliting}
\end{align}
 be
 an exact sequence of
$\Zhu(\W_k(\fing))$-modules.
We need to show that 
this sequence splits.

Recall that $\Ls{\bar \lam}=\Ms{\bar \lam}$ for $\lam\in Pr^k_{non-deg}$.
%$J_E=\Ann_{\mc{Z}(\fing)}(E)\subset \ker \gamma_{\lam}$.
By Lemma \ref{Lem:extension}
there exists
an exact sequence
\begin{align}\label{eq:non-spltiting-O}
0\ra \Ls{\bar \lam}\ra N\ra \Ls{\bar \lam}\ra 0
\end{align}of $\fing$-modules
that gives the exact sequence
(\ref{eq:non-spliting})
by applying the functor $H_0(\finn_-,?)$.
%such that $H_0(\finn, N\* \C_{\chi})\cong E$.
Since   $\Ann_{U(\fing)} \Ls{\bar \lam}=U(\fing)\ker \gamma_{\lam}$
we have
\begin{align}
 \Ann_{U(\fing)} N=U(\fing)\Ann_{\mc{Z}(\fing)}E.
\label{eq:annhilater}
\end{align}
On the other hand,
by applying the exact functor $Y\*_{\mc{Z}(\fing)}?$
to (\ref{eq:non-spliting})
we obtain the exact sequence
of $\Zhu(\Vs{k})$-modules
\begin{align*}
 0\ra Y_{\gamma_{\bar \lam}}\ra 
Y\*_{\mc{Z}(\fing)}E\ra Y_{\gamma_{\bar \lam}}\ra 0
\end{align*}
by Theorem \ref{Th:equivalence}.
It follows similarly that
\begin{align}
 \Ann_{U(\fing)}(Y\*_{\mc{Z}(\fing)}E)
=U(\fing)\Ann_{\mc{Z}(\fing)}E.
%=\Ann_{U(\fing)}N.
\label{eq:annihilator-is-the-same}
\end{align}
From (\ref{eq:annhilater}) and (\ref{eq:annihilator-is-the-same}),
it follows that
$N$ is  a module over 
$\Zhu(\Vs{k})$ as well,
and
\eqref{eq:non-spltiting-O}
is 
an  exact sequence
of $\Zhu(\Vs{k})$-modules.
Therefore by applying
 the functor
$\U(\Vs{k})\*_{\U(\Vs{k})_{\leq 0}}?$
to \eqref{eq:non-spltiting-O} 
 we
 obtain an exact sequence
\begin{align}
 0\ra \Irr{\lam}\ra M_{\Vs{k}}(N)\ra 
\Irr{\lam}\ra 0
\label{eq:this-splits}
\end{align}
of $\Vs{k}$-modules
by Proposition  \ref{Pro:induction}.
Here
the map
$\Irr{\lam}\ra M_{\Vs{k}}(N)$ is injective since $L(\lam)$ is simple.
Now, thanks to 
 Gorelik and Kac \cite{GorKac0905},
an admissible $\affg$-module
does not admit a non-trivial self-extension.
Therefore (\ref{eq:this-splits})
must split.
Restricting  (\ref{eq:this-splits})
we see that
(\ref{eq:non-spltiting-O})
splits, and therefore,
(\ref{eq:non-spliting})
splits as well.
This completes the proof.
\end{proof}
Let $\BGG_k$ be the full subcategory of category $\BGG$
of $\affg$ consisting of modules of level $k$,
which can be regarded as a full subcategory of $\Vg{k}\on{-Mod}$.
Let $H_-^{0}(?): \BGG_k\ra \W^k(\fing)\Mod$
be the  quantized Drinfeld-Sokolov ``$-$''-reduction functor
\cite{FKW92}.

Recall the following result.
\begin{Th}[\cite{Ara07}]
\label{Th:character-of-W-modules}
Let $k$ be any complex number.
\begin{enumerate}
 \item The functor $H_-^{0}(?): \BGG_k\ra \W^k(\fing)\Mod$
is exact.
\item For $\lam\in \dual{\affh}_k$,
$H^0_-(M(\lam))\cong \VermaW{\gamma_{\bar \lam}}$.
\item For $\lam\in \dual{\affh}_k$,
$H^0_-(L(\lam))\cong \begin{cases}
\IrrW{\gamma_{\bar \lam}}
		      &\text{if $\bar \lam$ is anti-dominant},\\
0&\text{otherwise.}
		     \end{cases}
$

\end{enumerate}
 
\end{Th}

Let $[M(\lam): L(\mu)]$ (resp.\ $[\VermaW{\gamma}: \IrrW{\gamma'}]$)
be
the
multiplicity of 
$L(\mu)$ (resp.\ $\IrrW{\gamma'}$) in the local composition factor
of $M(\lam)$
(resp.\ in the local composition factor of  $\VermaW{\gamma}$).

 \begin{Co}\label{Co:multiplicity-formula}
Let $\lam,\mu \in \dual{\affh}_k$ and
suppose that
$\bar \mu$ is anti-dominant.
Then 
\begin{align*}
[\VermaW{\gamma_{\bar \lam}}: \IrrW{\gamma_{\bar \mu}}]=
[M(\lam): L(\mu)].
\end{align*} \end{Co}
 \begin{proof}
Since $\ch M(\lam)=\sum_{\mu}[M(\lam): L(\mu)]
\ch L(\mu)$
we have
\begin{align*}
 \ch \VermaW{\gamma_{\bar \lam}}=\sum_{\mu\in \affW(\lam)\circ \lam\atop \bar \mu
\text{ is anti-dominant}}[M(\lam):L(\mu))]\ch \IrrW{\gamma_{\bar \mu}}.
\end{align*}
It remains to
observe that
if  $\mu,\mu'\in \affW(\lam)\circ \lam$,
$\gamma_{\bar \mu}=\gamma_{\bar \mu'}$,
and $\bar \mu$ and $\bar \mu'$ are both anti-dominant
then $\mu=\mu'$.
 \end{proof}

\begin{Th}\label{Th:rationality}
 Let $k$ be a non-degenerate admissible number
for $\affg$.
The simple vertex operator algebra
$\W_k(\fing)$  is
rational.
\end{Th}
\begin{proof}
By Theorem
 \ref{Th:simple=image-of-simple},
it is sufficient  to show that
\begin{align*}
\Ext^1_{\W_k(\fing)\on{-Mod}}(\IrrW{\gamma},\IrrW{\gamma'})=0\quad 
\text{for }
\IrrW{\gamma}, \IrrW{\gamma'}\in
\on{Irr}(\W_k(\fing)).
\end{align*}By Theorem \ref{Th:classification-of-simlple-modules-of-minimal-W}
we can
write
$\gamma=\gamma_{\bar \lam}$,
$\gamma=\gamma_{\bar \lam'}$
with
$\lam,\lam'\in Pr^k_{non-deg}$.
Let
\begin{align}
 0\ra \IrrW{\gamma'}\ra N\ra \IrrW{\gamma}\ra 0
\label{eq:exact-sequence-splits?}
\end{align}
be an exact sequence of $\W_k(\fing)$-modules.

Let $\Delta_{\gamma}$ be  the $L_0$-eigenvalue of the
lowest weight vector $v_{\gamma}$ of $\IrrW{\gamma}$,
which is a rational number.
Suppose that
$\Delta_{\gamma}<\Delta_{\gamma'}$,
and
choose a vector $v\in N_{\Delta_{\gamma}}$
such that the image of $v$ in $\IrrW{\gamma}$
is  $v_{\gamma}$.
Then there is a $\W^k(\fing)$-module
homomorphism
$\VermaW{\gamma}\ra N$ that sends the highest weight vector
of $\VermaW{\gamma}$ to $v$.
If (\ref{eq:exact-sequence-splits?})
is non-splitting,
 $N$ 
 must coincide with the image of $\VermaW{\gamma}$.
In particular,
$[\VermaW{\gamma}:\IrrW{\gamma'}]\ne 0$.
By 
Corollary \ref{Co:multiplicity-formula},
this is equivalent to
$[M(\lam): L(\lam')]\ne 0$.
This forces that
$\lam=\lam'$
since both $\lam$ and $\lam'$ are  dominant weighs of $\affg$.
This contradicts the assumption that $\Delta_{\gamma}<\Delta_{\gamma'}$.

By applying the duality functor $D(?)$ to
 (\ref{eq:exact-sequence-splits?}),
we see that the same argument applies 
to show that
 $\Ext_{\W_k(\fing)\Mod}^1(\IrrW{\gamma},\IrrW{\gamma'})=0$
in
the case $\Delta_{\gamma}>\Delta_{\gamma'}$.

Finally, suppose that
$\Delta_{\gamma}=\Delta_{\gamma'}=:\Delta$.
Then we have the exact sequence
\begin{align*}
 0\ra \IrrW{\gamma'}_{\Delta}\ra N_{\Delta}\ra \IrrW{\gamma}_{\Delta}\ra 0.
\end{align*}
The semisimplicity  of $\Zhu(\W_k(\fing))$ (Theorem \ref{Th:semi-simplicity-of-Zhu})
implies that
the above sequence splits.
Therefore (\ref{eq:exact-sequence-splits?})
splits as well.
This completes the proof.
\end{proof}

Main Theorem follows immediately from Theorems
\ref{Th:classification-of-simlple-modules-of-minimal-W},
\ref{Th:semi-simplicity-of-Zhu}
and \ref{Th:rationality}.
\qed

\newcommand{\etalchar}[1]{$^{#1}$}

\end{document}